\documentclass{amsart}
\usepackage[english]{babel}
\usepackage{latexsym,amsmath,amsthm,amssymb,graphics,graphicx,amscd,amsfonts,csquotes,tabu, booktabs,xparse,bbm}
\usepackage{stmaryrd}
\usepackage{cancel}
\usepackage[arrow, matrix, curve]{xy}
\usepackage{tikz} 
\usepackage{scalerel,stackengine}             
\newcommand{\R}{\mathbb{R}}

\usepackage{hyperref}
\DeclareMathOperator{\pr}{pr}

\newcommand{\mx}{\mathfrak{X}}
\newcommand{\dr}{\mathbf{d}}
\newcommand{\ldr}[1]{{{\pounds}}_{#1}}
\newcommand{\ip}[1]{{\mathbf{i}}_{#1}}
\newcommand{\an}[1]{\arrowvert_{#1}}
\DeclareMathOperator{\Id}{Id}

\DeclareMathOperator{\Hom}{Hom}
\newcommand{\lb}{\llbracket} \newcommand{\rb}{\rrbracket}
\newtheorem{theorem}{Theorem}[section]
\newtheorem{lemma}[theorem]{Lemma}
\newtheorem{proposition}[theorem]{Proposition}
\newtheorem{corollary}[theorem]{Corollary}
\newtheorem{definition}[theorem]{Definition} \theoremstyle{definition}
\newtheorem{example}[theorem]{Example}
\newtheorem{remark}[theorem]{Remark}

\newtheorem*{theorem*}{Theorem} 

\stackMath
\newcommand\reallywidehat[1]{%
\savestack{\tmpbox}{\stretchto{%
  \scaleto{%
    \scalerel*[\widthof{\ensuremath{#1}}]{\kern-.6pt\bigwedge\kern-.6pt}%
    {\rule[-\textheight/2]{1ex}{\textheight}}%WIDTH-LIMITED BIG WEDGE
  }{\textheight}% 
}{0.6ex}}%
\stackon[1pt]{#1}{\tmpbox}%
}

\stackMath
\newcommand\reallywidetilde[1]{%
\savestack{\tmpbox}{\stretchto{%
  \scaleto{%
    \scalerel*[\widthof{\ensuremath{#1}}]{\kern-.6pt\sim\kern-.6pt}%
    {\rule[-\textheight/2]{1ex}{\textheight}}%WIDTH-LIMITED BIG WEDGE
  }{\textheight}% 
}{0.9ex}}%
\stackon[1pt]{#1}{\tmpbox}%
}

\NewDocumentCommand\party{mmO{}}{\frac{\partial^{#3} #1}{\partial^{#3} #2}}

\title{Natural lifts of Dorfman brackets}

\author{M. Jotz Lean}
\address{Mathematisches Institut, Georg-August Universit\"at G\"ottingen.} 
\email{madeleine.jotz-lean@mathematik.uni-goettingen.de}
\author{C. Kirchhoff-Lukat}
\address{Departement Wiskunde, KU Leuven.}
\email{charlotte.kirchhofflukat@kuleuven.be}

\begin{document}
\maketitle

\begin{abstract}
  In this note we prove that, for a vector bundle $E$ over a manifold
  $M$, a Dorfman bracket on $TM\oplus E^*$ anchored by $\pr_{TM}$ and
  with $E$ a vector bundle over $M$, is equivalent to a lift from
  $\Gamma(TM\oplus E^*)$ to linear sections of $TE\oplus T^*E\to E$,
  that intertwines the given Dorfman bracket with the Courant-Dorfman
  bracket on sections of $TE\oplus T^*E$.

  This shows a universality of the Courant-Dorfman bracket,
  and allows us to caracterise twistings and symmetries of transitive
  Dorfman brackets via the corresponding lifts.

  \emph{This version of the manuscript was accepted for publication in
  ``Advances in Theoretical and Mathematical Physics'' in October 2018.}
\end{abstract}

\tableofcontents 

\section{Introduction}
Theodore Courant and his adviser Alan Weinstein defined\footnote{See
  \cite{Kosmann-Schwarzbach13} for a nice exposition of the history of
  Courant algebroids.}  the Courant bracket in 1990
\cite{CoWe88,Courant90a}: an $\R$-bilinear, skew-symmetric bracket on
sections of $TM\oplus T^*M$ that satisfies the Jacobi identity up to
an exact one-form.  Irene Dorfman independently introduced that
structure in her definition and study of Dirac structures in the
context of infinite dimensional Hamiltonian structures
\cite{Dorfman87}.  Then Liu, Weinstein and Xu discovered in the late
nineties that this bracket on sections of $TM\oplus T^*M$ is in fact a
particular, ``standard'' example of a Courant algebroid, when they
defined the later notion and proved that the bicrossproduct of any Lie
bialgebroid can be understood as a special type of Courant algebroid
\cite{LiWeXu97}.

 Nowadays, for a smooth manifold $M$, the standard
Courant algebroid structure on $TM\oplus T^*M$ is often defined using
the Courant-Dorfman bracket on $TM\oplus T^*M$: an $\R$-bilinear
bracket on sections of $TM \oplus T^*M$, that is not skew-symmetric
but satisfies a Jacobi identity written in Leibniz form (see
\cite{Roytenberg99,Uchino02}).  The two brackets are equivalent in the
sense that the Courant bracket is the skew-symmetrisation of the
Courant-Dorfman bracket.

In the context of Courant algebroids and Dirac structures, the
Courant-Dorfman bracket plays an important role in the generalised
geometry developped first by Nigel Hitchin, Marco Gualtieri (see
\cite{Hitchin03,Gualtieri03}). It also enters the theoretical
physics literature in this context: $TM\oplus T^*M$-generalised
geometry turns out to provide a convenient description for the
low-energy effective theory of closed string theory referred to as
\emph{double field theory} (see for instance \cite{HuZw09, HoHuZw10}).

Subsequently, the low-energy effective theories of the conjectured
M-theory were linked to Dorfman brackets on vector bundles of the form
$TM \oplus \wedge^{k_1} T^*M \oplus \dots \oplus \wedge^{k_l} T^*M$
(see \cite{Hull07}).

In all of these applications, Dorfman brackets encode infinitesimal
gauge transformations of the physical theory. Gauge transformations or
gauge invariances are redundancies in the mathematical description of
the theory (not to be confused with physical symmetries) -- the
physical results are invariant under the application of such
transformations. For example, general relativity, a theory of
four-dimensional smooth manifolds with Lorentzian metrics, is
invariant under diffeomorphisms. The Lie algebra of the diffeomorphism
group on a smooth manifold is given by the Lie derivatives $\ldr{X}$ for $X \in
\mx(M)$, so the Lie bracket (the simplest example of a Dorfman
bracket) gives the infinitesimal gauge transformations of general
relativity.  

Similarly, the theory described by the $TM\oplus T^*M$-generalised
geometry, which is a theory of a metric and a 2-form on a smooth
manifold $M$, is invariant under the semi-direct product of the
diffeomorphism group with the (additive) group of closed two-forms
$\operatorname{Diff}(M) \ltimes \Omega^2_{cl}(M)$ -- the physics of
this theory only depends on the exterior derivative of the two-form.
The Lie algebra of this group of \emph{generalised diffeomorphisms} is
precisely given by elements $\lb (X,\xi), \cdot \rb, (X,\xi) \in
\Gamma(TM \oplus T^*M)$, so the Courant-Dorfman bracket encodes the
infinitesimal gauge transformations of this more extended theory.
This principle is repeated in the M-theory examples.

Dorfman-type brackets on $TM \oplus \wedge^{k_1} T^*M \oplus \dots
\oplus \wedge^{k_l} T^*M$ and generalisations are studied in great
detail in \cite{Baraglia12} under the name of \emph{closed-form
  algebroids} as a special case of the general concept of
\emph{Leibniz algebroid}.  Leibniz algebroids are the natural
generalisation of Lie algebroids, where the bracket is no longer
required to be antisymmetric, but still satisfies a form of Jacobi
identity.

This paper studies Leibniz algebroids on vector bundles of the form
$TM \oplus E^*$, where $E\rightarrow M$ is some smooth vector bundle,
in the context of double vector bundles; more specifically the
standard VB-Courant algebroid $TE \oplus T^*E$ over the vector bundle $E$.  We call Leibniz
brackets of this type \emph{Dorfman brackets}, since they constitute
the most direct generalisation of the original Courant-Dorfman bracket
on $TM\oplus T^*M$.

Section \ref{lin_sec} characterises linear sections of $TE\oplus T^*E$
in terms of certain derivations of its core $E \oplus T^*M$. Linear
sections of $TE\oplus T^*E$ form a locally free sheaf over $M$ and are
thus sections of a vector bundle $\widehat{E}\rightarrow M$, the
so-called fat vector bundle. $\widehat{E}$ is in fact isomorphic as a
vector bundle to the Omni-Lie algebroid $\operatorname{Der}(E^*)\oplus
J^1(E^*)$ studied in \cite{ChLi10,ChZhSh11}. Using our results on
linear sections, we can show that the $E^*$-valued Courant algebroid
structure on $\operatorname{Der}(E^*)\oplus J^1(E^*)$ is induced from
the standard Courant algebroid structure on $TE\oplus T^*E$. Note that
according to \cite{Jotz19b}, the VB-Courant algebroid $TE\oplus
T^*E\to E$ is equivalent to an $E^*$-Courant algebroid. In
\cite{ChLiSh10} the omni-Lie algebroid associated to $E^*$ is proved
to be an $E^*$-Courant algebroid. To our knowledge, those two $E^*$-
Courant algebroids have never be proved to coincide before.

Furthermore, these results are used in Section
\ref{sec:Dorfmanlifts} to establish the following main result (Theorem
\ref{super}), which shows that all Dorfman brackets on $TM\oplus E^*$
are intimately linked to the Courant-Dorfman bracket on
$TE\oplus T^*E$. Therefore, the Courant-Dorfman bracket can be seen as
universal in the family of the Dorfman brackets.
\begin{theorem*} 
  Let $\lb \cdot, \cdot \rb$ be any Dorfman bracket on $TM \oplus E^*$
  anchored by $\pr_{TM}$.  Then there exists an $\R$-linear map
  $\Xi\colon\Gamma(TM \oplus E^*) \rightarrow \Gamma_E^l(TE \oplus T^*E)$
  which satisfies
\begin{enumerate} 
\item If $\Phi_E\colon TE \oplus T^*E \rightarrow TM \oplus E^*$ is the
  projection in the double vector bundle $(TE\oplus T^*E; TM \oplus
  E^*, E; M)$ (see Section \ref{subsec:VBCA}),
\[ \Phi_E(\Xi(\nu)(e_m))=\nu(m) \] 
for all $\nu \in \Gamma(TM\oplus E^*)$, $e_m \in E_m$ and $m\in M$.
\item The lift is \emph{natural} in the sense that for all $\nu_1,\nu_2 \in
  \Gamma(TM\oplus E^*)$, we have:
  \[ \Xi\lb \nu_1,\nu_2 \rb = \lb \Xi(\nu_1),\Xi(\nu_2)\rb \] where
  the bracket on the right-hand side is the Courant-Dorfman bracket on
  sections of $TE\oplus T^*E\to E$.
\end{enumerate} 
\end{theorem*} 

We compare this to results obtained in \cite{Jotz18a,ChLi10,ChZhSh11}. 

Section \ref{sec:examples} explores the most important examples of
such natural lifts, and sections \ref{sec:twists} and \ref{sec:symm}
describe twistings and internal symmetries of Dorfman brackets in
light of the double vector bundle context.

\subsection*{Notation and conventions}
We write $p_M\colon TM\to M$, $q_E\colon E\to M$ for vector bundle
projections.  We write $\langle\cdot\,,\cdot\rangle$ for the canonical
pairing of a vector bundle with its dual; i.e.~$\langle
e_m,\varepsilon_m\rangle=\varepsilon_m(e_m)$ for $e_m\in E$ and
$\varepsilon_m\in E^*$. We use several different pairings; in general,
which pairing is used is clear from its arguments.  Given a section
$\varepsilon$ of $E^*$, we write $\ell_\varepsilon\colon E\to \R$ for
the linear function associated to it, i.e.~the function defined by
$e_m\mapsto \langle \varepsilon(m), e_m\rangle$ for all $e_m\in E$.
We denote by $\iota_E\colon E\to E\oplus T^*M$ the canonical inclusion. 

Let $M$ be a smooth manifold. We denote by $\mx(M)$ and
$\Omega^1(M)$ the sheaves of smooth sections of the tangent and
the cotangent bundle, respectively. For an arbitrary vector bundle
$E\to M$, the sheaf of sections of $E$ is written
$\Gamma(E)$.

\section{Preliminaries on Courant algebroids, Dorfman brackets, dull brackets and Dorfman connections}
An anchored vector bundle is a vector bundle $Q\to M$ endowed with a
vector bundle morphism $\rho_Q\colon Q\to TM$ over the
identity.
Consider an anchored vector bundle $(\mathsf E\to M, \rho)$ and a
vector bundle $V$ over the same base $M$ together with a morphism
$\widetilde\rho\colon\mathsf E\to \operatorname{Der}(V)$, such
that the symbol of $\widetilde\rho(e)$ is $\rho(e)\in \mx(M)$ for all
$e\in \Gamma(\mathsf E)$.  Assume that $\mathsf E$ is paired with
itself via a nondegenerate pairing $\langle\cdot\,,\cdot\rangle\colon \mathsf
E\times_M \mathsf E\to V$ with values in $V$.  Then $\mathsf E\to M$ is a \textbf{
  Courant algebroid with pairing in $V$}
if
$\mathsf E$ is in addition equipped with an $\R$-bilinear bracket
$\lb\cdot\,,\cdot\rb$ on the smooth sections $\Gamma(\mathsf E)$ such
that the following conditions are satisfied:
\begin{enumerate}
\item $\lb e_1, \lb e_2, e_3\rb\rb= \lb \lb e_1, e_2\rb, e_3\rb+ \lb
  e_2, \lb e_1, e_3\rb\rb$,
\item $\widetilde\rho(e_1 )\langle e_2, e_3\rangle= \langle\lb e_1,
  e_2\rb, e_3\rangle + \langle e_2, \lb e_1 , e_3\rb\rangle$,
%\item $[e_1, e_1] = \frac{1}{2}\mathcal D\langle e_1 , e_1\rangle$,
\item $\lb e_1, e_2\rb+\lb e_2, e_1\rb =\mathcal D\langle e_1 ,
  e_2\rangle$,
\item
$\widetilde\rho\lb e_1, e_2\rb = [\widetilde\rho(e_1), \widetilde\rho(e_2)]$
\end{enumerate}
for all $e_1, e_2, e_3\in\Gamma(\mathsf E)$ and $f\in C^\infty(M)$,
where $\mathcal D\colon \Gamma(V)\to \Gamma(\mathsf E)$ is defined by
$\langle \mathcal Dv, e\rangle=\widetilde\rho(e)(v)$ for all
$v\in\Gamma(V)$.  Note that \begin{enumerate}\setcounter{enumi}{4}
\item $\lb e_1, f e_2\rb= f \lb e_1 , e_2\rb+ (\rho(e_1 )f )e_2$
\end{enumerate}
for $e_1,e_2\in\Gamma(\mathsf E)$ and $f\in C^\infty(M)$ follows from
(2). If $V=\mathbb R\times M\to M$ is the trivial
bundle, then $\mathcal D =\rho^*\circ\dr \colon
C^\infty(M)\to\Gamma(\mathsf E)$, where $\mathsf E$ is identified with
$\mathsf E^*$ via the pairing.  The quadruple $(\mathsf E\to M, \rho,
\langle\cdot\,,\cdot\rangle, \lb\cdot\,,\cdot\rb)$ is then a
\textbf{Courant algebroid} \cite{LiWeXu97,Roytenberg99}; then
$\widetilde\rho=\rho$ and (4) follows from (2) and the nondegeneracy of
the pairing (see also \cite{Uchino02}).  Finally note that Courant
algebroids with a pairing in a vector bundle $E$ were defined in
\cite{ChLiSh10} and called \emph{$E$-Courant algebroids}.

\begin{example}\label{ex_pontryagin}\cite{Courant90a}
The direct sum $TM\oplus T^*M$ 
endowed with the projection on $TM$ as anchor map, $\rho=\pr_{TM}$, 
the symmetric bracket 
$\langle\cdot\,,\cdot\rangle$
given by 
\begin{equation}
\langle(v_m,\theta_m), (w_m,\eta_m)\rangle=\theta_m(w_m)+\eta_m(v_m)
\label{sym_bracket}
\end{equation}
for all $m\in M$, $v_m,w_m\in T_mM$ and $\alpha_m,\beta_m\in T_m^*M$
and the \textbf{Courant-Dorfman bracket} 
given by 
\begin{align}
\lb (X,\theta), (Y,\eta)\rb&=\left([X,Y], \ldr{X}\eta-\ip{Y}\dr\theta\right)\label{wrong_bracket}
\end{align}
for all $(X,\theta), (Y, \eta)\in\Gamma(TM\oplus T^*M)$,
yield the standard example of  a Courant algebroid, which is often called 
the \textbf{standard Courant algebroid over $M$}. 
The map $\mathcal D\colon  C^\infty(M)\to \Gamma(TM\oplus T^*M)$
is given by $\mathcal D f=(0, \dr f)$.
We are here particularly interested in the standard Courant algebroid
over the total space of a vector bundle. 
\medskip

Next we define dull
algebroids and Leibniz algebroids.
\end{example}
\begin{definition}
\begin{enumerate}
\item \cite{Jotz18a} A \textbf{dull algebroid} is an anchored vector
  bundle $(Q\to M,\rho)$ endowed with a bracket $\lb\cdot\,,\cdot\rb$ on $\Gamma(Q)$ satisfying $ \rho\lb
  q_1,q_2\rb=[\rho(q_1),\rho(q_2)]$, and
   the Leibniz identity in both terms
\begin{equation*} \lb f_1 q_1, f_2 q_2\rb= f_1 f_2\lb q_1, q_2\rb+
  f_1\rho(q_1)( f_2)q_2- f_2\rho(q_2)( f_1)q_1
\end{equation*}
for all $ f_1, f_2\in C^\infty(M)$, $q_1, q_2\in\Gamma(Q)$.
\item \cite{Baraglia12} A \textbf{Leibniz algebroid} is an anchored vector bundle
  $(Q\to M,\rho)$ endowed with a
  bracket $\lb\cdot\,,\cdot\rb$ on $\Gamma(Q)$ with $\lb q_1, f q_2
  \rb= f \lb q_1, q_2 \rb + \rho(q_1)(f) q_2\ \forall f \in
  C^{\infty}(M), q_1,q_2 \in \Gamma(Q)$, and satisfying the Jacobi
  identity in Leibniz form
\begin{equation*} 
\lb q_1, \lb q_2, q_3\rb\rb=\lb\lb q_1,q_2\rb, q_3\rb+\lb q_2, \lb q_1,q_3\rb\rb
\end{equation*}
for all $q_1, q_2, q_3\in\Gamma(Q)$.
\item A Leibniz algebroid $E'$ is
  \textbf{transitive} if the anchor $\rho\colon E' \rightarrow TM$ is
  surjective \cite{Baraglia12}. Then the Leibniz algebroid can be written $E'=TM \oplus E^*$ with $\rho=
  \pr_{TM}$ and $E \rightarrow M$ a vector bundle.  We call its bracket $\lb
  \cdot,\cdot \rb$ a \textbf{Dorfman bracket}\footnote{Occasionally
    the term \enquote{Dorfman bracket} is used for the bracket of
    arbitrary Leibniz algebroids in the literature, but in this paper
    it will exclusively refer to the case where the anchor is
    surjective and the underlying vector bundle is split.}.
\item A transitive Leibniz algebroid is
  \textbf{split} if there is a section $\sigma\colon TM \rightarrow
  E'$ of the anchor map such that $\sigma(\mx(M))$ is closed under
  the Leibniz bracket \cite{Baraglia12}.
\end{enumerate}
\end{definition}

First note that the definition of the Leibniz algebroid implies \cite{Baraglia12}
\[ \rho\lb q_1,q_2 \rb = [\rho(q_1),\rho(q_2)]\quad \text{ for all }
q_1,q_2 \in \Gamma(Q). \] 
Any split transitive Leibniz algebroid $E'$ forms a split short exact
sequence of vector bundles:
\begin{equation} 
0 \rightarrow E^* \hookrightarrow E' \stackrel{\rho}{\rightarrow} TM \rightarrow 0 
\end{equation} 
with $E^* = \ker \rho$. The splitting map $\sigma: TM \rightarrow E'$
induces an isomorphism $E' \cong TM \oplus E^*$. Since
$\sigma(\mx(M))$ is closed under the Leibniz bracket and $\rho \circ
\sigma= \operatorname{id}_{TM}$, we have $\lb \sigma(X), \sigma(Y)
\rb= \sigma[X,Y]$.  Thus, if we use $\sigma$ to define the isomorphism
$E'\to TM \oplus E^*$, we obtain a Dorfman bracket with the property
\begin{equation} 
  \lb (X,0),(Y,0) \rb \stackrel{\text{def}}{=} \lb \sigma(X),\sigma(Y) \rb = \sigma[X,Y]= ([X,Y],0) 
\end{equation} 
Correspondingly, we call a Dorfman bracket \emph{split} precisely if
it has this property.  \medskip

\medskip
Consider a dull algebroid $(Q,\rho,\lb\cdot\,,\cdot\rb)$. Then the bracket
can be dualised to a map 
\begin{equation*}
  \Delta\colon \Gamma(Q)\times\Gamma(Q^*)\to\Gamma(Q^*),
\quad \rho(q)\langle q', \tau\rangle=\langle \lb q,q'\rb,\tau\rangle
  +\langle q', \Delta_q\tau\rangle
\end{equation*} 
for all $q,q'\in\Gamma(Q)$ and $\tau\in\Gamma(Q^*)$.  The map $\Delta$
is then a \textbf{Dorfman ($Q$-)connection on $Q^*$} \cite{Jotz18a},
i.e.~an $\R$-bilinear map with
\begin{enumerate}
\item $\Delta_{ f q}\tau= f\Delta_q\tau+\langle q, \tau\rangle \cdot \rho^*\dr f$,
\item $\Delta_q( f \tau)= f\Delta_q\tau+\rho(q)(f)\tau$ and 
\item $\Delta_q(\rho^*\dr f)=\rho^*\dr(\ldr{\rho(q)})$
\end{enumerate}
 for all $ f\in C^\infty(M)$, $q,q'\in\Gamma(Q)$, $\tau\in\Gamma(Q^*)$.
The curvature of $\Delta$ is the map
  $R_\Delta\colon \Gamma(Q)\times\Gamma(Q)\to\Gamma(Q^*\otimes Q^*)$ defined
  on $q,q'\in\Gamma(Q)$ by
  $R_\Delta(q,q'):=\Delta_q\Delta_{q'}-\Delta_{q'}\Delta_q-\Delta_{\lb q,q'\rb}$.
For all $ f\in C^\infty(M)$ and $q_1,q_2,q_3\in \Gamma(Q)$, $\tau\in \Gamma(Q^*)$, we have 
\[\langle R_\Delta(q_1,q_2)\tau,q_3\rangle=\langle\lb\lb
q_1,q_2\rb,q_3\rb+\lb q_2,\lb q_1,q_3\rb\rb-\lb q_1,\lb q_2,q_3\rb\rb,
\tau\rangle.
\]

\medskip

Consider a Dorfman bracket 
$\lb \cdot\,,\cdot\rb\colon \Gamma(Q)\times\Gamma(Q)\to \Gamma(Q)$. Its dual map is
\[\mathcal D\colon \Gamma(Q)\to \operatorname{Der}(Q^*),
\]
defined by 
$\rho(q)\langle q', \tau\rangle=\langle q', \mathcal
D_{q}\tau\rangle+\langle \lb q,q'\rb,
\tau\rangle$
for all $q,q'\in \Gamma(Q)$ and
$\tau\in\Gamma(Q^*)$.
The Jacobi identity in Leibniz form for $\lb\cdot\,,\cdot\rb$ is equivalent to
\begin{equation}\label{eq_jac_dual}
\mathcal D_{q_1}\circ\mathcal D_{q_2}-\mathcal D_{q_2}\circ\mathcal D_{q_1}=\mathcal D_{\lb q_1,q_2\rb}
\end{equation}
for all $q_1,q_2\in\Gamma(Q)$. 

$\mathcal{D}$ allows the extension of the Dorfman bracket to all
tensor bundles of $Q$ via the Leibniz rule. In the theoretical physics
applications, this operation is called the \emph{generalised Lie
  derivative} due to its Lie algebra property.

\begin{example}\label{CD_ex}
  The bracket of a Courant algebroid $\mathsf E$ is a Dorfman
  bracket. Using the nondegenerate pairing to identify $\mathsf E$ with
  its dual, we find that $\mathcal D$ is in this case the ``adjoint
  action'': 
$\mathcal D_e=\lb e, \cdot\rb$
for $e\in\Gamma(\mathsf E)$.
\end{example}

\begin{example} \label{ex:forms} On any vector bundle of the form $TM
  \oplus E^*$ with $E=\wedge^{k_1} TM \oplus \dots \oplus \wedge^{k_l}
  TM$, there is a Dorfman bracket
  \begin{equation} \label{equ:formbr} \lb (X, \alpha), (Y,\beta) \rb =
    [X,Y] + \ldr{X} \beta - i_Y \dr \alpha \quad \text{ for } \quad (X,\alpha),
    (Y,\beta) \in \Gamma(TM\oplus E^*)
\end{equation} 
For simplicity of notation, consider the special case $TM \oplus
\wedge^k T^*M$ for the rest of this example -- the more general case
works in the same way.  Let $(T,\theta) \in \Gamma(\wedge^k TM \oplus
T^*M)$. Then we have
\begin{align} 
  \left\langle\mathcal{D}_{(X,\alpha)} (T,\theta), (Y,\beta) \right\rangle &= X\left\langle(T,\theta),(Y,\beta)\right\rangle - \left\langle \lb (X, \alpha), (Y,\beta) \rb, (T,\theta) \right\rangle \nonumber\\
  &= \left\langle \ldr{X} \theta,Y \right\rangle + \left\langle \ldr{X}T,
    \beta
  \right\rangle + \left\langle \ip{Y} \dr \alpha, T \right\rangle \nonumber \\
  &= \left\langle(\ldr{X} T, \ldr{X} \theta + (-1)^k \dr \alpha (T,\cdot)), Y+ \beta\right\rangle\nonumber
\end{align} 
which shows $\mathcal{D}_{(X,\alpha)} (T,\theta) = (\ldr{X} T, \ldr{X} \theta
+ (-1)^k \dr \alpha(T,\cdot))$.
\end{example} 

\begin{example} 
  \cite{Baraglia12} extensively discusses a generalisation of example
  \ref{ex:forms}, so-called closed-form Leibniz algebroids. All
  commonly studied examples of Dorfman brackets belong to this class
  of Leibniz algebroids.

  In addition to the terms in \eqref{equ:formbr}, closed form
  algebroids can for example contain terms that mix different degrees
  of differential forms:
\begin{equation} 
\lb (0;\alpha_k,0,0),(0;0,\beta_j,0) \rb = (-1)^{(k-1)j} (0;0,0,\dr \alpha_k \wedge \beta_j) 
\end{equation} 
for the Dorfman bracket on $TM \oplus \wedge^k T^*M \oplus \wedge^j T^*M \oplus \wedge^{k+j+1} T^*M$. 

Terms of this type correspond to terms of the following form in $\mathcal{D}$: 
\begin{align} 
&\left\langle\mathcal{D}_{(0;\alpha_k,0,0)} (T_k,T_j,T_{k+j+1};\theta), (Y;\beta_k,\beta_j,\beta_{j+k+1}) \right\rangle\nonumber \\
 &= - \left\langle\lb (0;\alpha_k,0,0),(Y;\beta_k,\beta_j,\beta_{j+k+1})\rb, (T_k,T_j,T_{k+j+1}; \theta) \right\rangle \nonumber \\ 
&= \left\langle(0;\ip{Y}\dr \alpha_k,0,(-1)^{(k-1)j +1}\dr \alpha_k \wedge \beta_j), (T_k,0,T_{k+j+1};0) \right\rangle \nonumber \\ 
&= \left\langle(0,(-1)^{(k-1)j +1}T_{k+j+1}\neg \dr \alpha_k,0;(-1)^k \ip{T_k} \dr \alpha),(Y;\beta_k,\beta_j,\beta_{k+j+1})\right\rangle\nonumber \end{align} 
and therefore
\begin{equation}\label{equ:mixterm}
\mathcal{D}_{(0;\alpha_k,0,0)} (0,0,T_{k+j+1};0) =(-1)^{(k-1)j +1} (0,T_{k+j+1}\neg \dr \alpha_k,0;0),
\end{equation}
where $\neg$ denotes contraction over the first (in this case) $(k+1)$ indices. 
\end{example} 

\begin{example} 
  A more complex example of \emph{closed-form algebroid} underlies the
  so-called $E_7$-exceptional generalised geometry (see
  \cite{Baraglia12, Hull07}).  
The vector bundle
\begin{equation} 
TM \oplus \wedge^2 T^*M \oplus \wedge^5 T^*M \oplus (\wedge^7 T^*M \otimes T^*M) 
\end{equation} 
carries a natural $E_7 \times \R^*$-structure and the Dorfman bracket (see \cite{Baraglia12})
\begin{align*} 
&\lb (X;\alpha_2,\alpha_5,u),(Y;\beta_2,\beta_5,v) \rb \\
&= ([X,Y];\ldr{X} \beta_2 - \ip{Y} \dr \alpha_2, \ldr{X} \beta_5- \ip{Y} \dr \alpha_5 + \dr \alpha_2 \wedge \beta_2,
 \ldr{X} v - \dr \alpha_2 \diamond \beta_5 + \dr \alpha_5 \diamond \beta_2), 
\end{align*} 
where $(\dr \alpha \diamond \beta)(X) = (\ip{X} \dr \alpha) \wedge
\beta$ for all $X \in \mx(M)$.  The dual map $\mathcal{D}$ is then
given as follows: $\mathcal{D}_{(X;\alpha_2,\alpha_5, u)}
(T_2,T_5, T_7 \otimes Z;\theta)$ is
\begin{align*} 
 (\ldr{X} T_2 - T_5 \neg \dr \alpha_2 + T_7 \neg \ip{Z} \dr \alpha_5,
\ldr{X} T_5 - T_7 \neg \ip{Z} \dr \alpha_2,  
0;\ldr{X}\theta +\dr \alpha_2 (T_2,\cdot) - \dr \alpha_5(T_5,\cdot))
\end{align*} 
\end{example} 

\begin{remark} 
  Note that all examples for Dorfman brackets in this paper are
  \emph{local}, i.e. their brackets are given in terms of differential
  operators in both components. There are non-local Leibniz
  algebroids, for an example see Appendix \ref{non-local}.
\end{remark}

\section{Linear sections of $TE\oplus T^*E\to E$}\label{lin_sec}
In this section, we recall some background notions on double vector
bundles. Then we describe the double
vector bundle structures on $TE$, on $T^*E$ and on $TE\oplus T^*E$,
for a vector bundle $E\rightarrow M$. In the last part of this section, we
characterise arbitrary linear sections of $TE\oplus T^*E \to E$ via a
certain class of derivations. 
\subsection{Double vector bundles and linear splittings}
We briefly recall the definitions of double vector bundles and of their
\textbf{linear} and \textbf{core} sections. We refer to
\cite{Pradines77,Mackenzie05,GrMe10a} for more detailed treatments.
A \textbf{double vector bundle} is a commutative square
\begin{equation*}
\begin{xy}
\xymatrix{
D \ar[r]^{\pi_B}\ar[d]_{\pi_A}& B\ar[d]^{q_B}\\
A\ar[r]_{q_A} & M}
\end{xy}
\end{equation*}
of vector bundles such that
\begin{equation}\label{add_add} (d_1+_Ad_2)+_B(d_3+_Ad_4)=(d_1+_Bd_3)+_A(d_2+_Bd_4)
\end{equation}
for $d_1,d_2,d_3,d_4\in D$ with $\pi_A(d_1)=\pi_A(d_2)$,
$\pi_A(d_3)=\pi_A(d_4)$ and $\pi_B(d_1)=\pi_B(d_3)$,
$\pi_B(d_2)=\pi_B(d_4)$.
Here, $+_A$ and $+_B$ are the additions in $D\to A$ and $D\to B$,
respectively.
The vector bundles $A$
and $B$ are called the \textbf{side bundles}. The \textbf{core} $C$ of
a double vector bundle is the intersection of the kernels of $\pi_A$
and of $\pi_B$. From \eqref{add_add} follows easily the existence of a
natural vector bundle structure on $C$ over
$M$. The
inclusion $C \hookrightarrow D$ is denoted by
$
C_m \ni c \longmapsto \overline{c} \in \pi_A^{-1}(0^A_m) \cap \pi_B^{-1}(0^B_m).
$

The space of sections
$\Gamma_B(D)$ is generated as a $C^{\infty}(B)$-module by two
special classes of sections (see \cite{Mackenzie11}), the
\textbf{linear} and the \textbf{core sections} which we now describe.
For a section $c\colon M \rightarrow C$, the corresponding
\textbf{core section} $c^\dagger\colon B \rightarrow D$ is defined as
$c^\dagger(b_m) = \widetilde{0}_{\vphantom{1}_{b_m}} +_A \overline{c(m)}$, $m \in M$, $b_m \in B_m$.
We denote the corresponding core section $A\to D$ by $c^\dagger$ also,
relying on the argument to distinguish between them. The space of core
sections of $D$ over $B$ is written as $\Gamma_B^c(D)$.

A section $\xi\in \Gamma_B(D)$ is called \textbf{linear} if $\xi\colon
B \rightarrow D$ is a bundle morphism from $B \rightarrow M$ to $D
\rightarrow A$ over a section $a\in\Gamma(A)$.  The space of linear
sections of $D$ over $B$ is denoted by $\Gamma^\ell_B(D)$.  Given
$\psi\in \Gamma(B^*\otimes C)$, there is a linear section
$\widetilde{\psi}\colon B\to D$ over the zero section $0^A\colon M\to A$
given by
$\widetilde{\psi}(b_m) = \widetilde{0}_{b_m}+_A \overline{\psi(b_m)}$.
We call $\widetilde{\psi}$ a \textbf{core-linear section}. 

\subsection{The tangent double and the cotangent double of a vector bundle}\label{cotangent_ex}
Let $q_E\colon E\to M$ be a vector bundle.  Then the tangent bundle
$TE$ has two vector bundle structures; one as the tangent bundle of
the manifold $E$, and the second as a vector bundle over $TM$. The
structure maps of $TE\to TM$ are the derivatives of the structure maps
of $E\to M$.
\begin{equation*}
\begin{xy}
\xymatrix{
TE \ar[d]_{Tq_E}\ar[r]^{p_E}& E\ar[d]^{q_E}\\
 TM\ar[r]_{p_M}& M}
\end{xy}
\end{equation*} 
The space $TE$ is a double vector bundle with core bundle
$E \to M$. The map $\bar{}\,\colon E\to p_E^{-1}(0^E)\cap
(Tq_E)^{-1}(0^{TM})$ sends $e_m\in E_m$ to $\bar
e_m=\left.\frac{d}{dt}\right\vert_{t=0}te_m\in T_{0^E_m}E$.
Hence the core vector field corresponding to $e \in \Gamma(E)$ is the
vertical lift $e^{\uparrow}\colon E \to TE$, i.e.~the vector field with
flow $\phi^{e^\uparrow}\colon E\times \R\to E$, $\phi^{e^{\uparrow}}_t(e'_m)=e'_m+te(m)$. An
element of $\Gamma^\ell_E(TE)=\mx^\ell(E)$ is called a \textbf{linear
  vector field}. It is well-known (see e.g.~\cite{Mackenzie05}) that a
linear vector field $\xi\in\mx^l(E)$ covering $X\in\mx(M)$ corresponds
to a derivation $D\colon \Gamma(E) \to \Gamma(E)$ over $X\in
\mx(M)$. The precise correspondence is
given by 
\begin{equation}\label{ableitungen}
\xi(\ell_{\varepsilon}) 
= \ell_{D^*(\varepsilon)} \,\,\,\, \text{ and }  \,\,\, \xi(q_E^*f)= q_E^*(X(f))
\end{equation}
for all $\varepsilon\in\Gamma(E^*)$ and $f\in C^\infty(M)$, where
$D^*\colon\Gamma(E^*)\to\Gamma(E^*)$ is the dual derivation to $D$.
We write $\widehat D$ for the linear vector field in $\mx^l(E)$
corresponding in this manner to a derivation $D$ of $\Gamma(E)$. 
Given a derivation $D$ over $X\in\mx(M)$, the explicit formula for $\widehat D$ is 
\begin{equation}\label{explicit_hat_D}
\widehat D(e_m)=T_meX(m)+_E\left.\frac{d}{dt}\right\an{t=0}(e_m-tD(e)(m))
\end{equation}
for $e_m\in E$ and any $e\in\Gamma(E)$ such that $e(m)=e_m$. 
\bigskip

Dualizing $TE$ over $E$, we get 
the double vector bundle 
\begin{align*}
\begin{xy}
\xymatrix{
T^*E\ar[r]^{c_E}\ar[d]_{r_E} &E\ar[d]^{q_E}\\
E^*\ar[r]_{q_{E^*}}&M
}\end{xy}.
\end{align*}
The map $r_E$ is given as follows. For $\theta_{e_m}$, 
$r_E(\theta_{e_m})\in E^*_m$, 
\[\langle r_E(\theta_{e_m}), e'_m\rangle=\left\langle \theta_{e_m},
  \left.\frac{d}{dt}\right\an{t=0}e_m+te'_m\right\rangle
\]
for all $e'_m\in E_m$.
The addition in $T^*E\to E^*$ is defined 
as follows. If $\theta_{e_m}$ and $\omega_{e_m'}$ are such that 
$r_E(\theta_{e_m})=r_E(\omega_{e_m'})=\varepsilon_m\in E^*_m$, then the sum
$\theta_{e_m}+_{r_E}\omega_{e_m'}\in T_{e_m+e_m'}^*E$
is given by 
\[\langle  \theta_{e_m}+_{E^*}\omega_{e_m'}, v_{e_m}+_{TM}v_{e_m'}\rangle
=\langle  \theta_{e_m}, v_{e_m}\rangle+\langle\omega_{e_m'}, v_{e_m'}\rangle
\]
for all $v_{e_m}\in T_{e_m}E$, $v_{e'_m}\in T_{e'm}E$ such that
$(q_E)_*(v_{e_m})=(q_E)_*(v_{e_m'})$.  

For $\varepsilon\in\Gamma(E^*)$, the one-form $\dr\ell_\varepsilon$ is
linear over $\varepsilon$: we have
$r_E(\dr_{e_m}\ell_\varepsilon)=\varepsilon(m)$ for all $m\in M$ and
the sum $\dr_{e_m}\ell_\varepsilon+_{r_E}\dr_{e_m'}\ell_\varepsilon$
equals $\dr_{e_m+e_m'}\ell_\varepsilon$. For $\theta\in\Omega^1(M)$,
the one-form $q_E^*\theta$ is a core section of $TE\to E$:
$r_E((q_E^*\theta)(e_m))=0^{E^*}_m$, and for 
$\phi\in\Gamma(\operatorname{Hom}(E,T^*M))$ the core-linear section $\widetilde\phi\in\Gamma_E^l(T^*E)$ 
is given by $\widetilde\phi(e_m)=(T_{e_m}q_E)^*\phi(e_m)$ for all $e_m\in E$.  The vector space $T^*_{e_m}E$ is
spanned by $\dr_{e_m}\ell_\varepsilon$ and $\dr_{e_m}(q_E^* f)$ for
all $\varepsilon\in\Gamma(E^*)$ and $ f\in C^\infty(M)$. Finally note that
 $\dr\ell_{f\varepsilon}=q_E^*\dr\ell\varepsilon+\widetilde{\varepsilon\otimes \dr f}$
for all $\varepsilon\in\Gamma(E^*)$ and $f\in C^\infty(M)$.

 By taking the direct sum over $E$ of $TE$ and $T^*E$, we get
a double vector bundle
\begin{align*}
\begin{xy}
\xymatrix{
TE\oplus T^*E\ar[r]^{\quad \pi_E}\ar[d]_{\Phi_E} &E\ar[d]^{q_E}\\
TM\oplus E^*\ar[r]_{\quad q_{TM\oplus E^*}}&M
}\end{xy}
\end{align*}
with side projection $\Phi_E=(q_E)_*\oplus r_E$ and core $E\oplus
T^*M$.  In the following, for any section $(e,\theta)$ of $E\oplus
T^*M$, the vertical section
$(e,\theta)^\uparrow\in\Gamma_E(T^{q_E}E\oplus (T^{q_E}E)^\circ)$ is
the pair defined by
\begin{equation}\label{def_of_vert_pair}
(e,\theta)^\uparrow(e_m')=\left(\left.\frac{d}{dt}\right\an{t=0}e_m'+te(m), (T_{e'_m}q_E)^t\theta(m)\right)
\end{equation}
for all $e_m'\in E$.  Note that by construction the vertical sections 
$(e,\theta)^\uparrow$ are core sections of $TE\oplus T^*E$ as a vector
bundle over $E$.

\medskip

The standard Courant algebroid structure over $E$ is linear and 
\begin{equation*}
\begin{xy}
  \xymatrix{
    TE\oplus T^*E\ar[rr]^{\Phi_E:=({q_E}_*, r_E)}\ar[d]_{\pi_E}&& TM\oplus E^*\ar[d]\\
    E\ar[rr]_{q_E}&&M }
\end{xy}
\end{equation*} is a VB-Courant algebroid (\cite{Li-Bland12}, see also \cite{Jotz19b}) with base $E$
and side $TM\oplus E^*\to M$, and 
with core $E\oplus T^*M\to M$.

The anchor $\Theta=\pr_{TE}\colon TE\oplus T^*E\to TE$ restricts to
the map $\partial_{E}=\pr_{E}\colon E\oplus T^*M\to E$ on the cores,
and defines an anchor $\rho_{TM\oplus E^*}=\pr_{TM}\colon TM\oplus
E^*\to TM$ on the side.  In other words, the anchor of
$(e,\theta)^\uparrow$ is $e^\uparrow\in \mx^c(E)$ and if
$\chi$ is a linear section of $TE\oplus T^*E\to E$
over $(X,\epsilon)\in\Gamma(TM\oplus E^*)$, the anchor
$\Theta(\chi)\in \mx^l(E)$ is linear over
$X$.

\subsection{The first jet bundle of a vector bundle.}\label{jet}
For convenience of the exposition in the next section and later on in
the paper, we recall here some basic facts about the first jet bundle
of a vector bundle, and we set some notations.
\medskip

  The first jet bundle
$J^1E$ of a vector bundle $E$ over $M$ is the space
$\{\eta_m\in\operatorname{Hom}(T_mM, T_{e_m}E) \mid m\in M,\, e_m\in
E_m\}$.  It has a projection to $\pr_E\colon J^1E\to E$ to $E$, $
\eta_m\in\operatorname{Hom}(T_mM, T_{e_m}E) \mapsto e_m$ and a
projection to $\pr\colon J^1E\to M$ to $M$, $\eta_m\mapsto m$.  This
second projection is the projection of a vector bundle structure over
$M$; for $\eta_m\in\operatorname{Hom}(T_mM, T_{e_m}E)$ and
$\mu_m\in\operatorname{Hom}(T_mM, T_{e'_m}E)$, we have
$\alpha\eta_m+\beta\mu_m\in\operatorname{Hom}(T_mM,T_{\alpha e_m+\beta
  e'_m}E)$,
\[
(\alpha\eta_m+\beta\mu_m)(v_m)=\alpha\eta_m(v_m)+_{TM}\beta\mu_m(v_m),
\]
where $+_{TM}$ is the addition in the tangent prolongation $TE\to TM$
of the vector bundle $E\to M$. 
For each $\phi_m\in \Hom(T_mM,E_m)$ we get an element $\iota(\phi_m)\in J^1E_m$
with $\pr_E(\iota(\phi_m))=0^E_m$, $\iota(\phi_m)(v_m)=T_m0^E(v_m)+\left.\frac{d}{dt}\right\vert_{t=0}t\phi_m(v_m)$.

Two elements $\eta_m\in\operatorname{Hom}(T_mM, T_{e_m}E)$ and
$\mu_m\in\operatorname{Hom}(T_mM, T_{e_m}E)$ differ by 
such an element $\phi_m\in\Hom(T_mM,E_m)$ and we have a short exact
sequence
\[ 0\longrightarrow \Hom(TM,E)\overset{\iota}{\longrightarrow}
J^1E \overset{\pr_E}{\longrightarrow} E\rightarrow 0
\]
of vector bundles over $M$. The corresponding sequence 
\[ 0\longrightarrow \Gamma(\Hom(TM,E))\overset{\iota}{\longrightarrow}
\Gamma(J^1E) \overset{\pr_E}{\longrightarrow} \Gamma(E)\rightarrow 0
\]
is canonically split by the map $j^1\colon\Gamma(E)\to\Gamma(J^1E)$,
$(j^1e)_m\in \operatorname{Hom}(T_mM,T_{e_m}E)$,
$(j^1e)_m(v_m)=T_me(v_m)$. In particular, given $m\in M$ and two
sections $e,e'\in\Gamma(E)$ with $e(m)=e'(m)$, we find
$(j^1e)_m=(j^1e')_m+\iota(\phi_m)$ for a $\phi_m\in \Hom(T_mM,E_m)$.
In other words, there is a canonical isomorphism
\begin{equation}\label{can_iso_jet}
\begin{split} 
\Gamma(J^1E)&\cong \Gamma(E)\oplus \Gamma(T^*M\otimes E),
\quad \mu \mapsto (\pr_{E}\mu,\mu-j^1(\pr_{E}\mu)). 
\end{split}
\end{equation} 

Furthermore, we have $j^1(e_1+e_2)=j^1e_1+j^1e_2$ and
$J^1(fe)=fj^1e+\iota(\dr f\otimes e)$ for all $e,e_1,e_2\in\Gamma(E)$
and $f\in C^\infty(M)$.  \medskip

Note finally that every element $\mu \in J^1_m(E)$ can be written
$\mu=(j^1e)_m$ with a local section $e \in
\Gamma(E)$. Furthermore, two local sections $e,e'\in\Gamma(E)$ define
the same element $(j^1e)_m=(j^1e')_m=:\mu\in J^1_m(E)$ if and only if
$T_me=T_me'$ as vector space morphisms $T_mM\to T_{e(m)}E$. That is,
$e(m)=e'(m)$ and $T_me(v_m)=T_me'(v_m)$ for all $m\in T_mM$.  The
later is equivalent to $v_m\langle \epsilon,e\rangle=(T_me
v_m)(\ell_\epsilon)=(T_me'v_m)(\ell_\epsilon) =v_m\langle
\epsilon,e'\rangle$ for all $v_m\in T_mM$ and all
$\epsilon\in\Gamma(E^*)$, and so to
\[\langle \dr_{\epsilon(m)}\ell_e,T_m\epsilon v_m\rangle=\langle
\dr_{\epsilon(m)}\ell_{e'},T_m\epsilon v_m\rangle\] for all $v_m\in T_mM$ and
all $\epsilon\in\Gamma(E^*)$.  Hence, $(j^1e)_m=(j^1e')_m$ if and only
if $\dr_{\epsilon}\ell_e=\dr_{\epsilon}\ell_{e'}$ for all $\epsilon\neq 0\in
E_m^*$; by continuity then
$\dr_{\epsilon}\ell_e=\dr_{\epsilon}\ell_{e'}$ for all $\epsilon\in E_m^*$.

\subsection{The $E^*$-valued Courant algebroid structure on the fat bundle 
$\widehat{E}$}\label{sec:fatb}
The space $\Gamma^l_E(TE\oplus T^*E)$ is a $C^\infty(M)$-module:
choose $f\in C^\infty(M)$ and $\chi\in\Gamma^l_E(TE\oplus T^*E)$ a
linear section over $\nu\in\Gamma(TM\oplus E^*)$. Then $q_E^*f\cdot
\chi$ is linear over $f\nu\in\Gamma(TM\oplus E^*)$.  The space
$\Gamma^l_E(TE\oplus T^*E)$ is a locally free and finitely generated
$C^{\infty}(M)$-module (this follows from the existence of local
splittings).  Hence, there is a vector bundle $\widehat{E}$ over $M$
such that $\Gamma^l_E(TE\oplus T^*E)$ is isomorphic to
$\Gamma(\widehat{E})$ as $C^{\infty}(M)$-modules. The vector bundle
$\widehat{E}$ is called the \emph{fat vector bundle} defined by
$\Gamma^l_E(TE\oplus T^*E)$. We prove below that it is isomorphic to
$\operatorname{Der}(E^*)\oplus J^1(E^*)$, where $\operatorname{Der}(E^*)$ is the
bundle of derivations on $E^*$, and $J^1(E^*)$ the first jet bundle.
\medskip

First recall that \eqref{ableitungen}
defines a
bijection between the linear vector fields $\mx^l(E)$ 
and $\Gamma(\operatorname{Der}(E^*))$. It is easy to see from
\eqref{ableitungen} that this bijection is a morphism of
$C^\infty(M)$-modules. Hence, the fat bundle defined by
$\mx^l(E)=\Gamma_E^l(TE)$ is the vector bundle $\operatorname{Der}(E^*)$.

Next note that  $\Gamma_E^l(T^*E)$ fits in the
short exact sequence 
\[ 0\longrightarrow\Gamma(\operatorname{Hom}(E,T^*M))\overset{\widetilde{\cdot}}{\longrightarrow}
\Gamma_E^l(T^*E)\overset{r_E}{\longrightarrow} \Gamma(E^*)\longrightarrow 0,
\] of $C^\infty(M)$-modules, where the second map sends $\phi\in
\Gamma(\operatorname{Hom}(E,T^*M))$ to the core-linear section
$\widetilde\phi\in\Gamma_E^l(T^*E)$, $\widetilde\phi(e)=(T_eq_E)^*\phi(e)$ for all $e\in E$,
and the third map sends $\theta\in \Gamma_E^l(T^*E)$ to its base
section $r_E\theta$ in $\Gamma(E^*)$. %  Therefore, the fat bundle
% $\reallywidehat{T^*E}$ defined by $\Gamma(\reallywidehat{T^*E})\simeq
% \Gamma_E^l(T^*E)$ fits in the short exact sequence
% \[ 0\longrightarrow \operatorname{Hom}(E,T^*M)\longrightarrow
% \reallywidehat{T^*E}\overset{\pr}{\longrightarrow} E^*\longrightarrow 0
% \] 
% of vector bundles over $M$. Given $m\in M$, we denote by
% ${\rm ev}_m\colon \Gamma_E^l(T^*E)\to \reallywidehat{T^*E}_m$ 
% the evaluation map at $m\in M$.
We define $\Psi\colon \Gamma(J^1E^*)\to \Gamma^l_ET^*E)$ by
$\Psi(j^1\epsilon)=\dr \ell_\epsilon$ for $\epsilon\in\Gamma(E^*)$ and
$\Psi(\iota\phi)=\widetilde{\phi^*}\in\Gamma_E^l(T^*E)$ for
$\phi\in\Gamma(\operatorname{Hom}(TM,E^*))$.
 The map $\Psi$ is
$C^\infty(M)$-linear  and we get the following commutative diagram
of morphisms of $C^\infty(M)$-modules
\begin{equation*}
\begin{xy}
  \xymatrix{
   0\ar[r]& \Gamma(\operatorname{Hom}(TM,E^*))\ar[d]_{(\cdot)^*}\ar[r]^{\quad \iota}&\Gamma(J^1E^*)\ar[d]^{\Psi}\ar[r]^{\pr_{E^*}}&\Gamma(E^*)\ar[r]\ar[d]^{\Id}&0\\
   0\ar[r]& \Gamma(\operatorname{Hom}(E,T^*M))\ar[r]_{\quad \widetilde{\cdot}}&\Gamma_E^l(T^*E)\ar[r]_{r_E}&\Gamma(E^*)\ar[r]&0\\
}
\end{xy}
\end{equation*}
with short exact sequences in the top and bottom rows. Since the left
and right vertical arrows are isomorphisms, $\Psi$ is an isomorphism
by the five lemma.
% We check that $\Psi$ is
% well-defined and injective. Choose $\epsilon,\epsilon'\in\Gamma(E^*)$.
% Then by Section \ref{jet}, $(j^1\epsilon)_m=(j^1\epsilon')_m$ if and
% only if $\dr_{e}\ell_\epsilon=\dr_{e}\ell_{\epsilon'}$ for all $e\in
% E_m$.  Since $\pr({\rm ev}_m(\dr\ell_\epsilon)-{\rm
%   ev}_m(\dr\ell_{\epsilon'}))=\epsilon(m)-\epsilon'(m)=0$, we have
% ${\rm ev}_m(\dr\ell_\epsilon)-{\rm
%   ev}_m(\dr\ell_{\epsilon'})=\widetilde{\phi_m}$ for a $\phi_m\in
% \operatorname{Hom}(E_m,T_m^*M)$.  Therefore,
% $(\dr\ell_\epsilon-\dr\ell_{\epsilon'})\an{E_m}=\widetilde{\phi_m}$.
% Hence $\dr_{e}\ell_\epsilon=\dr_{e}\ell_{\epsilon'}$ for all $e\in
% E_m$ is equivalent to $\phi_m=0$ and so to ${\rm
%   ev}_m(\dr\ell_\epsilon)={\rm ev}_m(\dr\ell_{\epsilon'})$.
 %  Both the fat bundle
% $\reallywidehat{T^*E}$ and the first jet bundle $J^1(E^*)$ have rank
% $k+nk$ if $E$ is of rank $k$ and $M$ of dimension $n$. Hence, since
% $\psi$ is injective, it is an isomorphism.
Since $\Psi$ is an isomorphism of $C^\infty(M)$-modules, we obtain a vector bundle isomorphism 
$\psi\colon J^1E^*\to\reallywidehat{T^*E}$, where $\reallywidehat{T^*E}$ is the fat bundle defined by $\Gamma_E^l(T^*E)$.
Finally we obtain a vector bundle isomorphism
\begin{equation} 
\Theta\colon\operatorname{Der}(E^*)\oplus J^1(E^*)\rightarrow \widehat{E},\qquad 
(D_m,(j^1\epsilon)_m)\mapsto {\rm ev}_m\left(\widehat{D^*},\dr \ell_{\epsilon}\right).
\end{equation}

Recall that for a linear section $\chi\in\Gamma_E^l(TE\oplus T^*E)$,
there exists a section $\nu\in\Gamma(TM\oplus E^*)$ such that
$\pi_{TM\oplus E^*}\circ \chi = \nu \circ q_E$. The map $\chi \mapsto
\nu$ induces a short exact sequence of vector bundles
\begin{equation*}
  0 \longrightarrow E^*\otimes (E\oplus T^*M) \hookrightarrow \widehat{E} \longrightarrow TM\oplus E^* \longrightarrow 0.
\end{equation*}

Note that the restriction of the pairing on $TE\oplus T^*E$ to linear
sections of $TE\oplus T^*E$ defines a nondegenerate pairing on
$\widehat{E}$ with values in $E^*$.  Since the Courant bracket of
linear sections is again linear, the vector bundle $\widehat{E}$
inherits a Courant algebroid structure with pairing in $E^*$ (see
\cite{Jotz19b}). In particular, the Courant algebroid structure on
$TE\oplus T^*E$ defines a Leibniz bracket on sections of $
\operatorname{Der}(E^*)\oplus J^1(E^*)$ and a pairing with values in
$E^*$ on \[\left(\operatorname{Der}(E^*)\oplus
  J^1(E^*)\right)\times_M\left(\operatorname{Der}(E^*)\oplus J^1(E^*)
\right).\] This is called an \emph{Omni-Lie algebroid} in \cite{ChLi10},
see also \cite{ChZhSh11}.  The symmetric bilinear nondegenerate
pairing with values in $E^*$ on $\widehat E$ is given by $\langle
\Theta(D(m)),\Theta((j^1\epsilon)_m+\iota\phi_m)\rangle =\langle \widehat{D^*},
\dr\ell_\epsilon+\widetilde\phi\rangle (m)=D(\epsilon)(m)+\phi^*(X)(m)$
for $D$ a derivation with symbol $X\in\mx(M)$.  Here, the second
term is the evaluation at $m\in M$ of 
the linear function $\ell_{D\epsilon+\phi^*X}$, when identified
with $D\epsilon+\phi^*X\in\Gamma(E^*)$. Hence, the
corresponding symmetric bilinear nondegenerate pairing with values in
$E^*$ on $J^1(E^*)\oplus \operatorname{Der}(E^*)$ is given by $
\left<D_m,(j^1\epsilon)_m+\iota\phi_m\right>=D_m(\epsilon)+\phi(X)(m) $ for $\epsilon\in
\Gamma(E^*)$, $\phi\in\Gamma(\operatorname{Hom}(TM,E^*))$ and  $D_m\in \mathcal{D}_m(E^*)$ with symbol $X\in\mx(M)$.

\subsection{Linear sections of $TE\oplus T^*E\to
  E$} \label{subsec:VBCA} In this section we build on the techniques
summarised in Section \ref{cotangent_ex} and we prove original results on linear
sections of $TE\oplus T^*E\to E$. Those results will be the basis of
our main theorem in Section \ref{sec:Dorfmanlifts}.

We consider a linear section $\chi\in\Gamma^l_E(TE\oplus T^*E)$ over a
pair $(X,\varepsilon)\in\Gamma(TM\oplus E^*)$.  Given a section
$e\in\Gamma(E)$, the difference
$\chi(e(m))-(T_meX(m),\dr_{e(m)}\ell_\varepsilon)$ projects to $e(m)$
in $E$ and to $0_m\in TM\oplus E^*$ and we can define a section
$D_\chi(e,0)\colon M\to E\oplus T^*M$ by
\[\chi(e(m))-(T_meX(m),\dr_{e(m)}\ell_\varepsilon)=-D_\chi(e,0)^\uparrow(e(m))
\]
for all $m\in M$. By construction and the scalar multiplication in the
fibers of $TE\oplus T^*E\to TM\oplus E^*$, we get
$D_\chi(re,0)=rD_\chi(e,0)$ for a real number $r\in\R$, and
$D_\chi(e_1+e_2,0)=D_\chi(e_1,0)+D_\chi(e_2,0)$ for $e_1,e_2\in \Gamma(E)$.  For a smooth
function $f\in C^\infty(M)$, we have $\chi((fe)(m))=\chi(f(m)e(m))$
and
\[\left(T_m(fe)X(m),\dr_{f(m)e(m)}\ell_\varepsilon\right)
=\left(T_m(f(m)e)X(m)+(X(f)e)^\uparrow(f(m)e(m)),\dr_{f(m)e(m)}\ell_\varepsilon\right).\]
Hence, we find that 
\begin{equation}\label{D_chi_der}
D_\chi(fe,0)=fD_\chi(e,0)+(X(f)e,0).
\end{equation}
Now we set $D_\chi\colon\Gamma(E\oplus T^*M)\to\Gamma(E\oplus T^*M)$,
$D_\chi(e,\theta)=D_\chi(e,0)+(0,\ldr{X}\theta)$. \eqref{D_chi_der}
and Theorem \ref{prop_1_chi} below shows that $D_\chi$ is a smooth
derivation. We have found the following result:

\begin{theorem}\label{linear_are_der_thm}
Let $\chi$ be a linear section of $TE\oplus T^*E\to E$ over a pair $(X,\varepsilon)\in\Gamma(TM\oplus E^*)$.
Then there exists a unique derivation 
$D_\chi\colon\Gamma(E\oplus T^*M)\to\Gamma(E\oplus T^*M)$ with symbol $X\in\mx(M)$ and which satisfies 
\begin{enumerate}
\item $D_\chi(e,\theta)=D_\chi(e,0)+(0,\ldr{X}\theta)$ and 
\item $\chi(e(m))=(T_meX(m),\dr_{e(m)}\ell_\varepsilon)-D_\chi(e,0)^\uparrow(e(m))$,
\end{enumerate}
for all $e\in\Gamma(E)$ and $\theta\in\Omega^1(M)$.
\end{theorem}
Conversely, given a pair $(X,\varepsilon)\in\Gamma(TM\oplus E^*)$ and
a smooth derivation $D\colon \Gamma(E\oplus T^*M)\to\Gamma(E\oplus
T^*M)$ over $X\in\mx(M)$, we write $\chi_{\varepsilon,D}$ for the
linear section defined by
\[\chi_{\varepsilon,D}(e(m))=(T_meX(m),\dr_{e(m)}\ell_\varepsilon)-D(e,0)^\uparrow(e(m))
\]
for all $e\in\Gamma(E)$.  Note that (1) in the last theorem shows that
for each $\chi \in\Gamma_E^l(TE\oplus T^*E)$ there exist a
derivation $d_\chi\in\Gamma(\operatorname{Der}(E))$ and a tensor
$\phi_\chi\in \Gamma(E^*\otimes T^*M)$ with $D_\chi(e,0)=(d_\chi
e,\phi_\chi(e))$. More precisely, $d_\chi=\pr_E\circ D_\chi\circ
\iota_E\colon \Gamma(E)\to\Gamma(E)$ is a derivation of $E$ with
symbol $X$ and the vector bundle morphism is $\phi_\chi=\pr_{T^*M}\circ D_\chi\circ
\iota_E\colon E\to T^*M$. The linear
section $\chi$ can then be written
\begin{equation*} \label{equ:chi}
\chi=\left(\widehat{d_\chi},\dr\ell_\varepsilon-\widetilde{\phi_\chi}\right)
\end{equation*} 

\begin{remark}
With the results in Section \ref{sec:fatb}, we can phrase this
correspondence in terms of the bundle isomorphism
$\widehat{E}\cong\operatorname{Der}(E^*)\oplus J^1(E^*)$: 
$\chi=(\widehat{d_{\chi}},\dr
\ell_{\varepsilon}-\widetilde{\phi_{\chi}})\in
\Gamma(\widehat{E})$ corresponds to $(d_\chi, j^1\varepsilon-\iota(\phi_{\chi}))$ in
$\Gamma(\operatorname{Der}(E^*)\oplus J^1(E^*))$. % : Obviously, $d_{\chi}$ is
% simply the $\operatorname{Der}(E^*)$-part.  If $\mu \in \Gamma(J^1(E^*))$ is
% the corresponding section to the one-form part $\alpha=\dr
% \ell_{\varepsilon} - \widetilde{\phi_{\chi}}$, $\mu(m)=[\epsilon_m]_m\
% \forall m\in M$ for a family of sections $\epsilon_m\in \Gamma(E^*)$
% such that $\epsilon_m(m)=\varepsilon(m)$, and
% \begin{align*} 
% \dr \ell_{\epsilon_m}(e_m)&= \dr \ell_{\varepsilon}(e_m)-\widetilde{\phi_{\chi}}(e_m) \forall\ m \in M, \\ 
% \text{i.e. } \phi_{\chi}(e)(m) &= \dr \left<\varepsilon-\epsilon_m,e\right>\vert_m\ \forall e\in \Gamma(e), m\in M 
% \end{align*} 
% This is $C^{\infty}$-linear in $E$ because $(\varepsilon-\epsilon_m)(m)=0\ \forall m\in M$. 
% \medskip
\end{remark} 

We can use these results on linear sections to prove the following: 
\begin{theorem}\label{prop_1_chi}
  Let $\chi$ be a linear section of $TE\oplus T^*E \to E$ over
  $(X,\varepsilon)\in\Gamma(TM\oplus E^*)$.  The Courant-Dorfman
  bracket on sections of $TE\oplus T^*E\to E$ satisfies
\begin{equation*}
\left\lb\chi, \tau^\uparrow\right\rb=D_\chi\tau^\uparrow
\end{equation*}
and the pairing 
\begin{equation*}
\langle \chi, \tau^\uparrow\rangle=q_E^*\langle (X,\varepsilon), \tau\rangle.
\end{equation*}
for all $\tau\in\Gamma(E\oplus T^*M)$.
The anchor satisfies
$\pr_{TE}(\chi)=\widehat{d_\chi}$.

\end{theorem}
We prove the first identity in Appendix \ref{proofs}.  The second and
third identities follow immediately from (\ref{equ:chi}).

We now state our first main theorem.
\begin{theorem}\label{prop_2_chi}
  Choose two linear sections $\chi_1,\chi_2\in\Gamma^l_E(TE\oplus
  T^*E)$, over pairs $(X_1,\varepsilon_1),
  (X_2,\varepsilon_2)\in\Gamma(TM\oplus E^*)$.
Then we have 
\begin{equation}\label{prop_2_chi_eq}
\begin{split}
\lb \chi_1, \chi_2\rb&=\left(\reallywidehat{[d_{\chi_1},d_{\chi_2}]},
  \dr\ell_{\pr_{E^*}D_{\chi_1}^*(X_2,\varepsilon_2)}-\reallywidetilde{\pr_{T^*M}\circ[D_{\chi_1}, D_{\chi_2}]\circ\iota_E}\right)\\
&=\chi_{\pr_{E^*}D_{\chi_1}^*(X_2,\varepsilon_2), [D_{\chi_1}, D_{\chi_2}]}
\end{split}
\end{equation}
and 
$\langle \chi_1, \chi_2\rangle=\ell_{\pr_{E^*}(D_{\chi_1}^*(X_2,\varepsilon_2)+D_{\chi_2}^*(X_1,\varepsilon_1))}$.
\end{theorem}

The Theorem is again proved in Appendix \ref{proofs} and gives us an
expression for the induced $E^*$-valued Courant bracket on
$\operatorname{Der}(E^*)\oplus J^1(E^*)$:
\begin{corollary} Choose $d_{1},d_{2}\in
  \Gamma(\operatorname{Der}(E^*))$ with symbols $X_1, X_2\in\mx(M)$ and choose $\mu_1,\mu_2 \in \Gamma(J^1(E^*))$
  corresponding as in \eqref{can_iso_jet} to
  $(\varepsilon_1,\phi_1),(\varepsilon_2,\phi_2)\in\Gamma(E^*)\oplus
  \Gamma(T^*M\otimes E^*)$. Then
\begin{equation} 
\lb (d_{1},\mu_1),(d_{2},\mu_2)\rb=\left([d_{1},d_{2}], \ldr{d_{1}}\mu_2 - \ldr{d_{2}} \mu_1 + j^1\left<d_{2},\mu_1\right>\right), 
\end{equation} 
where the $\operatorname{Der}(E^*)$-Lie derivative on $J^1(E^*)$ is
defined in equation (19) of \cite{ChLi10}: 
\begin{equation*}
\ldr{d} \mu= \ldr{d} (\varepsilon,\phi)= (d\varepsilon, (\ldr{X} \circ \phi^* -\phi^*\circ d^*)^*)
\end{equation*}
where $d$ is a derivation of $E^*$ with symbol $X\in\mx(M)$ and $\mu=(\varepsilon,\phi)\in\Gamma(J^1E^*)\simeq\Gamma(E^*\oplus\operatorname{Hom}(TM,E^*))$.
  Thus, our
theorem proves that the $E^*$-valued Courant algebroid structure on
$\operatorname{Der}(E^*)\oplus J^1(E^*)$ given in \cite{ChLi10} is
precisely induced from $TE\oplus T^*E$ via the isomorphism $\Psi$ from
\ref{sec:fatb}.
\end{corollary} 

\begin{proof} 
  With the correspondence between $\Gamma(\widehat{E})$ and
  $\Gamma(\operatorname{Der}(E^*)\oplus
  \Gamma(J^1(E^*))=\Gamma(\operatorname{Der}(E^*)\oplus
  E^*\oplus \operatorname{Hom}(TM,E^*))$,
  $(d_i,(\varepsilon_i,\phi_i))$ corresponds to
  $\chi_i=\chi_{\varepsilon_i,D_{i}}$ with
  $D_i(e,0)=(d_i(e),-\phi_i^*(e))$. Then we have $\pr_{E^*}
  D_{\chi_1}^*(X_2,\varepsilon_2)=d_{1}(\varepsilon_2)+\phi_1(X_2)$
  as well as
\begin{equation*} 
\pr_{T^*M}([D_{\chi_1},D_{\chi_2}](e,0)) = -\phi_1^*(d^*_{2} e) + \phi_2^*(d^*_{1} e) - \ldr{X_1}(\phi^*_2(e)) +\ldr{X_2}(\phi^*_1(e)).
\end{equation*}
By the considerations in Section \ref{sec:fatb}, we have further
$\left<d_{2},\mu_1\right> = d_{2}\varepsilon_1 + \phi_1(X_2)$.  We get
using the isomorphisms $\Gamma(J^1E^*)\simeq
\Gamma(E^*\oplus\operatorname{Hom}(TM,E^*))$ and $\Gamma_E^l(TE\oplus
T^*E)\simeq\Gamma(J^1E^*\oplus \operatorname{Der}(E^*))$:
\begin{equation*}
\begin{split}
\lb (d_1,\mu_1), (&d_2,\mu_2)\rb=\lb(d_1,\varepsilon_1,\phi_1),(d_2,\varepsilon_2,\phi_2)\rb
=\left\lb \left(\widehat{d_1},\dr\ell_{\varepsilon_1}+\widetilde{\phi_1}\right),\left(\widehat{d_2},\dr\ell_{\varepsilon_2}+\widetilde{\phi_2}\right)\right\rb\\
&=\left\lb \chi_{\epsilon_1,D_1},\chi_{\epsilon_2,D_2}\right\rb=\chi_{(d_{1}(\varepsilon_2)+\phi_1(X_2)), [D_1,D_2]}\\
&=([d_1,d_2], d_{1}(\varepsilon_2)+\phi_1(X_2), \ldr{d_1}\phi_2-\ldr{d_2}\phi_1)\\
&=([d_1,d_2],0,0)+(0,d_2\varepsilon_1+\phi_1(X_2),0)+(0, \ldr{d_1}(\varepsilon_2,\phi_2)-\ldr{d_2}(\varepsilon_1,\phi_1))\\
&=([d_1,d_2],j^1\langle d_2,\mu_1\rangle+\ldr{d_1}\mu_2-\ldr{d_2}\mu_1).
\end{split}
\end{equation*}
\end{proof} 

Note that the derivation
$D_\chi$ defines as follows a derivation of
$\operatorname{Hom}(E,E\oplus T^*M)$:
$(D_\chi\varphi)(e)=D_\chi(\varphi(e))-\varphi(d_\chi(e)) $ for
all $e\in\Gamma(E)$.
\begin{corollary}\label{new_lemma}
In the situation of the preceding theorem,
the Courant-Dorfman bracket satisfies
 $\left\lb \chi, \widetilde\varphi\right\rb=\widetilde{D_\chi\varphi}$ for
  $\varphi\in\Gamma(\operatorname{Hom}(E,E\oplus T^*M))$. 
\end{corollary}

\begin{proof}The section $\varphi \in \Gamma(E^*\otimes(E\oplus T^*M))$
  can be written as $\varphi=(\phi_1,\phi_2^*)$, with $\phi_1\in
  \Gamma(E^*\otimes E)$ and $\phi_2 \in \Gamma(\operatorname{Hom}(TM,E^*))$. Furthermore, $\phi$ defines a section of
  $\operatorname{Der}(E^*)\oplus J^1(E^*)$: $\phi_1^*$ is a derivation
  of $E^*$
  with symbol $0\in\mx(M)$ and $\phi_2\simeq\iota\phi_2$ is a section of
  $J^1E^*$. Therefore $\widetilde{\phi}$ is simply the
  corresponding core-linear section under the correspondence outlined
  above. Choose  $\chi=(d,\mu)$ with $d$ a derivation of $E^*$ over $X\in\mx(M)$ and $\mu=j^1\varepsilon+\iota\psi\in\Gamma(J^1E^*)$. Then
the results above yield
\[ \lb (d,\mu), (\phi_1,\phi_2)\rb = \lb (d,\varepsilon,\psi), (\phi_1^*,0,\phi_2)\rb
=([d,\phi_1^*],(\ldr{X} \circ \phi_2^* - \phi_2^* \circ d^*)^*+\phi_1^*\circ\psi), \]  
which is easily seen to be $D_{\chi}\varphi$.
\end{proof}

\subsection{Linear closed $3$-forms}
Let $E$ be as usual a vector bundle over $M$. A $k$-form $H$ on $E$ is
\textbf{linear} if the induced vector bundle morphism $H^\sharp\colon
\oplus^{k-1}TE\to T^*E$ over the identity on $E$ is \emph{also} a
vector bundle morphism over a map $h\colon \oplus^{k-1}TM\to E^*$ on
the other side of the double vector bundles \cite{BuCa12}.

According to Proposition 1 in \cite{BuCa12}, a linear $k$-form $H\in\Omega^k(E)$
can be written 
\[ H=\dr\Lambda_\mu+\Lambda_\omega
\]
with $\mu\in\Omega^{k-1}(M,E^*)$ and $\omega\in\Omega^k(M,E^*)$.
Here, given $\omega\in\Omega^k(M,E^*)$, the $k$-form
$\Lambda_\omega\in\Omega^k(E)$ is given by
\[\Lambda_\omega(e_m)=(T_{e_m}q_E)^*(\langle \omega, e\rangle(m)),
\]
where $\langle \omega, e\rangle\in\Omega^k(M)$ is the obtained $k$-form on $M$.
Note that in the equation for $H$, we have $\mu=(-1)^{k-1}h$.

\begin{example}
  For instance, we have seen in \S\ref{cotangent_ex} that for
  $\varepsilon\in\Gamma(E^*)$, the 1-form
  $\dr\ell_\varepsilon\in\Omega^1(E)$ is linear. Since it projects to
  $\varepsilon\in\Gamma(E^*)$, we know that any linear 1-form on $E$
  can be written $\dr\ell_\varepsilon+\widetilde\phi$ for
  $\varepsilon\in\Gamma(E^*)=\Omega^0(M,E^*)$ and
  $\phi\in\Gamma(\Hom(E,T^*M))=\Omega^1(M,E^*)$.  An easy computation
  shows $\Lambda_\varepsilon=\ell_\varepsilon\in \Omega^0(E)=C^\infty(E)$ and
  $\Lambda_\phi=\widetilde\phi\in\Omega^1(E)$.
\end{example}

\begin{proposition} \label{closed_linear} Consider a linear $k$-form
  $H=\dr\Lambda_\mu+\Lambda_\omega$, with $\mu\in\Omega^{k-1}(M,E^*)$
  and $\omega\in\Omega^k(M,E^*)$.  Then $H$ is closed, $\dr H=0$, if and
  only if $\omega=0$.
\end{proposition}
\begin{proof}
  $H$ is closed if and only if $\Lambda_\omega$ is closed. It is enough
  to evaluate $\dr\Lambda_\omega$ on linear and core vector fields on
  $E$.  Take $k$ linear vector fields $\widehat{D_i}\in \mx^l(E)$ over
  $X_i\in \mx (M)$, $i=1,\ldots, k$, and one vertical vector field
  $e^\uparrow\in \mx^c(E)$.
 Then
\begin{equation*}
\begin{split}
  (\dr\Lambda_\omega)\left(\widehat{D_1}, \ldots, \widehat{D_{k}},
    e^\uparrow\right) =&\sum_{i=1}^{k}(-1)^{i+1}\widehat{D_i}\left(\Lambda_\omega\left(\widehat{D_1}, \ldots, \hat i
  \ldots, \widehat{D_{k}}, e^\uparrow\right)\right)\\
&+(-1)^{k}e^\uparrow\left(\Lambda_\omega\left(\widehat{D_1}, \ldots,\widehat{D_{k}}\right)\right)\\
&+\sum_{1\leq i\le j\leq k}(-1)^{i+j}\Lambda_\omega\left(\left[\widehat{D_i},\widehat{D_j}\right], \widehat{D_1}, \ldots, \hat i,
  \ldots, \hat j, \ldots,\widehat{D_{k-1}}, e^\uparrow\right)\\
&+\sum_{i=1}^{k}(-1)^{i+k}\Lambda_\omega\left(\left[\widehat{D_i},e^\uparrow\right], \widehat{D_1}, \ldots, \hat i,
  \ldots, \widehat{D_{k}}\right).
\end{split}
\end{equation*}
Since $\left[\widehat{D_i}, e^\uparrow\right]$ is again a vertical vector
field and $\Lambda_\omega$ vanishes on vertical vector fields, the
first, third and fourth terms of this sum all vanish.
The remaining term is 
\[(-1)^{k}e^\uparrow\left(\Lambda_\omega\left(\widehat{D_1}, \ldots, \widehat{D_{k}}\right)\right)
=(-1)^{k}q_E^*\langle \omega(X_1, \ldots, X_{k}), e\rangle.
\]
This is $0$ for all $X_1, \ldots, X_k\in\mx(M)$ and $e\in\Gamma(E)$ if
and only if $\omega=0$.
\end{proof}

In what follows, we will consider closed linear $3$-forms
$H=\dr\Lambda_\mu$ with $\mu\in \Omega^{2}(M, E^*)$ the base map of
$H^\sharp$. Let us compute the inner product of such a $3$-form with
two linear vector fields on $E$.

Recall that any linear vector field can be written $\widehat
D\in\mx^l(E)$ with a derivation $D\colon \Gamma(E) \to \Gamma(E)$ over $X\in\mx(M)$.
The derivation $D$ induces a derivation $D\colon\Omega^1(M,E^*)\to
\Omega^1(M,E^*)$ by
\[(D\omega)(Y)=D^*(\omega(Y))-\omega[X,Y]
\]
for all $\omega\in \Omega^1(M,E^*)$ and $Y\in \mx(M)$. In particular,
given a Dorfman bracket on sections of $TM\oplus E^*$, the linear
vector field $\pr_{TE}\Xi(\nu)$ equals $\widehat{\delta_\nu}$, where
$\nu$ is a section of $TM\oplus E^*$ and $\delta_\nu$ is the derivation
over $\pr_{TM}\nu$. We write $\delta_\nu$ for the induced derivation of 
$\Omega^1(M,E^*)$.
\begin{proposition}
  Choose $\mu\in\Omega^2(M,E^*)$. Let $\widehat{D_1}, \widehat{D_2},
  \widehat{ D}\in \mx^l(E)$ be linear vector fields over $X_1,X_2,
  X\in\mx(M)$ and let $e$ be a section of $E$. Then
\begin{equation}\label{inner1}
  \ip{\widehat{D_2}}\ip{\widehat{D_1}}\dr\Lambda_\mu=
  \dr\ell_{\ip{X_2}\ip{X_1}\mu}+\widetilde{D_{1}(\ip{X_2}\mu)}
-\widetilde{D_{2}(\ip{X_1}\mu)}-\widetilde{\ip{[X_1,X_2]}\mu}
\end{equation} 
and 
\begin{equation}\label{inner2}
\ip{e^\uparrow}\ip{\widehat{D}}\dr\Lambda_\mu=-q_E^*\langle \ip{X}\mu, e\rangle.
\end{equation}
\end{proposition}

\begin{proof}
We have for $e\in\Gamma(E)$:
\begin{align*}
  \ip{e^\uparrow}\ip{\widehat{D_2}}\ip{\widehat{D_1}}\dr\Lambda_\mu=\,&\widehat{D_1}\left(\Lambda_\mu\left(\widehat{D_2},
      e^\uparrow\right)\right)-\widehat{D_2}\left(\Lambda_\mu\left(\widehat{D_1}, e^\uparrow\right)\right)
  +e^\uparrow\left(\Lambda_\mu\left(\widehat{D_1},\widehat{D_2}\right)\right)\\
  &-\Lambda_\mu\left(\left[\widehat{D_1},\widehat{D_2}\right],
    e^\uparrow\right)+\Lambda_\mu\left(\left[\widehat{D_1},e^\uparrow\right], \widehat{D_2}\right)
  -\Lambda_\mu\left(\left[\widehat{D_2},e^\uparrow\right], \widehat{D_1}\right)\\
=\,& 0-0+e^\uparrow(\ell_{\mu(X_1,X_2)})-0+0-0=q_E^*\langle \mu(X_1,X_2), e \rangle.
\end{align*}
This shows that $\ip{\widehat{D_2}}\ip{\widehat{D_1}}\dr\Lambda_\mu=\dr\ell_{\ip{X_2}\ip{X_1}\mu}+\widetilde\phi$
for a section $\phi\in\Gamma(E^*\otimes T^*M)$.
We have then
\begin{align*}
  \ell_{\langle\phi, X_3\rangle}=\,&\langle \widetilde \phi,
  \widehat{D_3}\rangle
  =\ip{\widehat{D_3}}\ip{\widehat{D_2}}\ip{\widehat{D_1}}\dr\Lambda_\mu-\widehat{D_3}(\ell_{\ip{X_2}\ip{X_1}\mu})\\
  =\,&\widehat{D_1}\left(\Lambda_\mu\left(\widehat{D_2},
      \widehat{D_3}\right)\right)-\widehat{D_2}\left(\Lambda_\mu\left(\widehat{D_1},
      \widehat{D_3}\right)\right)+\cancel{\widehat{D_3}\left(\Lambda_\mu\left(\widehat{D_1},\widehat{D_2} \right)\right)}\\
  &-\Lambda_\mu\left(\left[\widehat{D_1},\widehat{D_2}\right],
    \widehat{D_3}\right)+\Lambda_\mu\left(\left[\widehat{D_1},\widehat{D_3}\right],
    \widehat{D_2}\right) -\Lambda_\mu\left(\left[\widehat{D_2},\widehat{D_3}\right],\widehat{D_1}\right)\\
  &-\cancel{\widehat{D_3}(\ell_{\ip{X_2}\ip{X_1}\mu})}\\
  =\,&\widehat{D_1}(\ell_{\mu(X_2,X_3)})-\widehat{D_2}(\ell_{\mu(X_1,X_3)})
  -\Lambda_\mu\left(\widehat{[D_1,D_2]},
    \widehat{D_3}\right)\\
&+\Lambda_\mu\left(\widehat{[D_1, D_3]},
    \widehat{D_2}\right) -\Lambda_\mu\left(\widehat{[D_2,D_3]},\widehat{D_1}\right)\\
  =\,&\ell_{D_{1}^*(\mu(X_2,X_3))}-\ell_{D_{2}^*(\mu(X_1,X_3))}-\ell_{\mu([X_1,X_2],
    X_3)}+\ell_{\mu([X_1,X_3], X_2)}-\ell_{\mu([X_2,X_3], X_1)}\\
  =\,&\ell_{\langle
    D_{1}(\ip{X_2}\mu)-D_{2}(\ip{X_1}\mu)-\ip{[X_1,X_2]}\mu,
    X_3\rangle}
\end{align*}
and we find \eqref{inner1}.
In order to prove \eqref{inner2}, we use the equation
\[\ip{e^\uparrow}\ip{\widehat{D_2}}\ip{\widehat{D_1}}\dr\Lambda_\mu
=q_E^*\langle \mu(X_1,X_2), e \rangle\]
 above to find that 
\[\ip{\widehat{D_2}}\ip{e^\uparrow}\ip{\widehat{D_1}}\dr\Lambda_\mu
=-\ip{\widehat{D_2}}(q_E^*\langle \mu(X_1), e \rangle)
\]
for all linear $\widehat{D_2}\in \mx^l(E)$.  Since
$\ip{{e'}^\uparrow}\ip{e^\uparrow}\ip{\widehat{D_1}}\dr\Lambda_\mu=0=-\ip{{e'}^\uparrow}(q_E^*\langle \mu(X_1), e
\rangle)$ for all $e'\in\Gamma(E)$, we find that
$\ip{e^\uparrow}\ip{\widehat{D_1}}\dr\Lambda_\mu=-q_E^*\langle \mu(X_1), e
\rangle$.
\end{proof}

 We will use the following lemma.
\begin{lemma}\label{inner_product_linear}
  Choose an element $\beta\in\Omega^1(M,E^*)$ and a linear vector
  field $\widehat D\in\mx^l(E)$ over $X\in \mx(M)$.  Then
  $\ip{\widehat D}\dr\Lambda_{\beta}$ is a linear $1$-form over
  $\beta(X)\in \Gamma(E^*)$. More precisely,
\[\ip{\widehat D}\dr\Lambda_{\beta}=-\dr\ell_{\beta}+\widetilde{D\beta}.
\] 
\end{lemma}

\begin{proof}
We have 
\[\ip{e^\uparrow}\ip{\widehat D}\dr\Lambda_{\beta}=\widehat D\langle \Lambda_{\beta}, e^\uparrow\rangle
-e^\uparrow\langle \Lambda_{\beta}, \widehat D\rangle-\langle \Lambda_{\beta}, (De)^\uparrow\rangle
=-e^\uparrow(\ell_{{\beta}(X)})=-q_E^*\langle e, {\beta}(X)\rangle.
\]
Therefore $\ip{\widetilde
  X}\dr\Lambda_{\beta}=-\dr\ell_{{\beta}(X)}+\widetilde \phi$ with a
section $\phi\in\Gamma(\operatorname{Hom}(E,T^*M))$ to be determined.
We have 
\begin{align*}
  \left\langle \widetilde \phi, \widehat{D'}\right\rangle&=\langle \ip{\widehat{D}}\dr\Lambda_{\beta}+\dr\ell_{{\beta}(X)}, \widehat{D'}\rangle\\
  &=\widehat D(\ell_{{\beta}(Y)})-\cancel{\widehat{D'}(\ell_{{\beta}(X)})}-\ell_{{\beta}[X,Y]}+\cancel{\widehat{D'}(\ell_{{\beta}(X)})}\\
  &=\ell_{D^*({\beta}(Y))-{\beta}[X,Y]}=\ell_{(D{\beta})(Y)}=\left\langle \widetilde{D{\beta}},
    \widehat{D'}\right\rangle
\end{align*}
for any linear vector field $\widehat{D'}\in \mx^l(E)$ over $Y\in\mx(M)$.
This shows that $\phi=D{\beta}$.
\end{proof}

\section{Dorfman brackets and natural lifts.} \label{sec:Dorfmanlifts} 
Consider now an $\R$-linear lift
\[\Xi\colon\Gamma(TM\oplus E^*)\to \Gamma_E^l(TE\oplus T^*E),
\]
sending each section $(X,\varepsilon)$ of $TM\oplus E^*$ to a linear section over
$(X,\varepsilon)$.
Then the lift defines an $\R$-linear map
\[\mathcal D\colon \Gamma(TM\oplus E^*)\to \operatorname{Der}(E\oplus T^*M),
\quad \Xi(X,\varepsilon)=\chi_{\varepsilon, \mathcal D_{(X,\varepsilon)}}.
\]
Consider the dual \[ \lb \cdot\,,\cdot\rb\colon \Gamma(TM\oplus
E^*)\times\Gamma(TM\oplus E^*)\to \Gamma(TM\oplus E^*)\] of $\mathcal
D$, written in bracket form and defined by
\[X\langle \nu, \tau\rangle=\langle \nu, \mathcal
D_{(X,\varepsilon)}\tau\rangle+\langle \lb (X,\varepsilon), \nu\rb,
\tau\rangle
\]
for all $(X,\varepsilon), \nu\in \Gamma(TM\oplus E^*)$ and
$\tau\in\Gamma(E\oplus T^*M)$. Any bracket defined in this manner is
$\R$-bilinear, anchored by $\pr_{TM}$ and satisfies a Leibniz identity
in its second component.
We easily get the following result.
\begin{proposition}\label{rough_lift}
Lifts \[\Xi\colon\Gamma(TM\oplus E^*)\to \Gamma_E^l(TE\oplus T^*E),
\]
sending each section $(X,\varepsilon)$ of $TM\oplus E^*$ to a linear section over
$(X,\varepsilon)$ are equivalent to $\R$-bilinear brackets $\lb
\cdot\,,\cdot\rb\colon \Gamma(TM\oplus E^*)\times\Gamma(TM\oplus
E^*)\to \Gamma(TM\oplus E^*)$, that are anchored by $\pr_{TM}$ and
satisfy a Leibniz identity in the second component.
\end{proposition}

  Define further the map $\delta\colon
\Gamma(TM\oplus E^*)\to \operatorname{Der}(E)$ by
\[\delta_\nu =\pr_E\circ\mathcal D_\nu\circ \iota_E.
\] 

As we have seen before, the lift $\Xi\colon
\Gamma(TM\oplus E^*)\to\Gamma^l_E(TE\oplus T^*E)$ can be written
\begin{equation}\label{general_lift}
\Xi(X,\varepsilon)(e_m)=\left(T_meX(m), \dr_{e_m}\ell_\varepsilon\right)-(\mathcal D_{(X,\varepsilon)}(e,0))^\uparrow(e_m),
\end{equation}
for any $e\in\Gamma(E)$ with $e(m)=e_m$, or 
\[\Xi(X,\varepsilon)=(\widehat{\delta_{(X,\varepsilon)}},\dr\ell_\varepsilon-\reallywidetilde{\pr_{T^*M}\mathcal
    D_{(X,\varepsilon)}\circ\iota_E}).
\]

In terms of sections of the Omni-Lie algebroid
$\operatorname{Der}(E^*)\oplus J^1(E^*)$, this says that anchored
$\R$-bilinear brackets on $TM\oplus E^*$ with Leibniz rule in the
second component are in one-to-one correspondence with splittings
\[\Xi\colon \Gamma(TM\oplus E^*) \rightarrow \Gamma(\operatorname{Der}(E^*)\oplus J^1(E^*)),\]
of the short exact sequence 
\[ 0 \rightarrow \Gamma(E^*\otimes(E\oplus T^*M))\rightarrow \Gamma(\operatorname{Der}(E^*)\oplus J^1(E^*))\rightarrow \Gamma(TM\oplus E^*)\rightarrow 0 \]

Note that in either description the map $\Xi$ is a map of sections
only, so its image will in general not span a sub-vector bundle of
$\widehat{E}\cong \operatorname{Der}(E^*)\oplus J^1(E^*)$.

We  prove the following theorem, which shows that a chosen lift as
above is \emph{natural}, if and only if the bracket
$\lb\cdot\,,\cdot\rb$ is a Dorfman bracket.
\begin{theorem}\label{super}
  Let $E$ be a smooth vector bundle over a manifold $M$.  Consider an
  $\R$-bilinear bracket $\lb\cdot\,,\cdot\rb$ on sections of $TM\oplus
  E^*$, that is anchored by $\pr_{TM}$ and satisfies the Leibniz
  identity in its second component.  Then $\lb\cdot\,,\cdot\rb$ is a
  Dorfman bracket if and only if the corresponding lift as in
  Proposition \ref{rough_lift} or \eqref{general_lift} is natural,
  i.e.~if and only if
\[\lb \Xi(\nu_1), \Xi(\nu_2)\rb=\Xi\lb \nu_1, \nu_2\rb
\] for all $\nu_1,\nu_2\in\Gamma(TM\oplus E^*)$, where the bracket on
the left-hand side is the Courant-Dorfman bracket.
\end{theorem}

The proof of this theorem follows from the general results in
\ref{subsec:VBCA} and is given in Appendix \ref{proofs2}.  Note that
the proof of this theorem can also be adapted in a straightforward
manner from the proof of the main theorem in \cite{Jotz18a} (see the
following remark); the only difference being that $\mathcal{D}$ is not
$C^{\infty}(M)$-linear in its lower entry. The proof in \cite{Jotz18a}
is however independent of this property.

\begin{remark}\label{sigma_Delta}
  Note that horizontal lifts $\sigma\colon\Gamma(TM\oplus E^*)\to
  \Gamma_E^l(TE\oplus T^*E)$ satisfying
  $\sigma(\nu_1+\nu_2)=\sigma(\nu_1)+\sigma(\nu_2)$ and
  $\sigma(f\cdot\nu)=q_E^*f\sigma(\nu)$ are called \emph{linear}. The
  horizontal lifts above are in general not linear; they are additive,
  but in general they are not $C^\infty(M)$-homogeneous.

  Linear horizontal lifts $\sigma\colon\Gamma(TM\oplus E^*)\to
  \Gamma_E^l(TE\oplus T^*E)$ were proved in \cite{Jotz18a} to be
  equivalent to dull brackets on sections of $TM\oplus E^*$, or
  equivalently to Dorfman connections $\Gamma(TM\oplus
  E^*)\times\Gamma(E\oplus T^*M)\to\Gamma(E\oplus T^*M)$.

Let $\Delta\colon \Gamma(TM\oplus E^*)\times\Gamma(E\oplus
T^*M)\to\Gamma(E\oplus T^*M)$ be a Dorfman connection and consider the
dual dull bracket $\lb\cdot\,,\cdot\rb_\Delta$.  Note that the map
\[\nabla\colon\Gamma(TM\oplus E^*)\times\Gamma(E)\to\Gamma(E), \qquad
\nabla_\nu e=\pr_E(\Delta_\nu(e,0))
\]
is a linear connection. Choose $\nu, \nu_1,\nu_2\in\Gamma(TM\oplus
E^*)$ and $\tau\in\Gamma(E\oplus T^*M)$. \cite{Jotz18a} proves the
following identities
\begin{enumerate}
\item $\left\langle \sigma^\Delta(\nu_1), \sigma^\Delta(\nu_2)\right\rangle=\ell_{
  \lb\nu_1, \nu_2\rb_\Delta+\lb\nu_2, \nu_1\rb_\Delta}$,
\item $\left\langle \sigma^\Delta(\nu),
    \tau^\uparrow\right\rangle=q_E^*\langle \nu, \tau\rangle$,
\item $\pr_{TE}\left(\sigma^\Delta(\nu)\right)=\widehat{\nabla_\nu}$
and $\pr_{TE}(\tau^\uparrow)=(\pr_E\tau)^\uparrow$,
\item $\left\lb \sigma^\Delta(\nu),
    \tau^\uparrow\right\rb=\left(\Delta_{\nu}\tau\right)^\uparrow$,
 \item $ \left\lb \sigma^\Delta(\nu_1), \sigma^\Delta(\nu_2)\right\rb =\sigma^\Delta(\lb \nu_1,
   \nu_2\rb_\Delta)-\widetilde{R_\Delta(\nu_1, \nu_2)\circ \iota_E}$.
\end{enumerate}
Those results could now be easily deduced from Theorems \ref{prop_1_chi} and
\ref{prop_2_chi}, as we deduce our main result Theorem \ref{super} from those.
\end{remark}

Here, we have the following result, a counterpart for Dorfman
brackets of the results described in Remark \ref{sigma_Delta}. Note that
since $\lb\cdot\,,\cdot\rb$ is anchored by $\pr_{TM}$, the sum $\lb
\nu_1,\nu_2\rb+\lb \nu_2,\nu_1\rb$ is a section of $E^*$ for all
$\nu_1,\nu_2\in\Gamma(TM\oplus E^*)$.

\begin{theorem}\label{main2}
  Let $\lb\cdot\,,\cdot\rb$ be a Dorfman bracket on sections of
  $TM\oplus E^*$.  For all $\nu,\nu_1,\nu_2\in\Gamma(TM\oplus E^*)$
  and $\tau\in\Gamma(E\oplus T^*M)$, we have
\begin{enumerate}
\item $\left\langle \Xi(\nu_1), \Xi(\nu_2)\right\rangle=\ell_{
  \lb\nu_1, \nu_2\rb+\lb\nu_2, \nu_1\rb}$ and  $\left\langle \Xi(\nu),
    \tau^\uparrow\right\rangle=q_E^*\langle \nu, \tau\rangle$,
\item $\pr_{TE}\left(\Xi(\nu)\right)=\widehat{\delta_\nu}$
and $\pr_{TE}(\tau^\uparrow)=(\pr_E\tau)^\uparrow$,
\item $\left\lb \Xi(\nu),
    \tau^\dagger\right\rb=\mathcal D_{\nu}\tau^\dagger$.
\end{enumerate}
\end{theorem}

\begin{proof}Those identities are all given by Theorems
  \ref{prop_1_chi} and \ref{prop_2_chi}.
\end{proof}

\subsection{Links to known results on Omni-Lie algebroids, on Dorfman
  connections and on the standard VB-Courant algebroid}

\cite{ChLi10, ChZhSh11} prove the following result on Lie
algebroid structures on subbundles of $TM\oplus E^*$ versus Dirac
structures inside the $E^*$-valued Courant-algebroid
$\mathcal{E}:=\operatorname{Der}(E^*)\oplus J^1(E^*)$.
Note that such a Dirac structure is called \emph{reducible} if its
projection to $TM\oplus E^*$ is surjective.
\begin{theorem}\textbf{\emph{(Theorem 3.7 in \cite{ChZhSh11})}} There
  is a one-to-one correspondence between reducible Dirac structures
  $L\subset \mathcal{E}$ and projective Lie algebroids
  $A\subset TM\oplus E^*$ such that $A$ is the quotient Lie algebroid
  of $L$. (As a Dirac structure, $L$ carries a Lie bracket induced by
  the Courant-Dorfman bracket.)
\end{theorem} 
A \emph{projective Lie algebroid} is a subbundle
$A\subset TM\oplus E^*$ with a Lie algebroid structure
$(A,[\cdot,\cdot]_A, \rho_A)$, with anchor given by
$\rho_A=\operatorname{pr}_{TM}\vert_A$. A \emph{reducible} Dirac
structure $L\subset \mathcal{E}$ is a Dirac structure the image of
which in $TM\oplus E^*$
under $\mathbf{b}\colon \mathcal{E}\rightarrow TM\oplus E^*$ is a regular
subbundle. The correspondence in the theorem is such that
$A=\mathbf{b}(L)$, and the Lie bracket is the quotient Lie bracket on
$A$ induced by the short exact sequence
\[ 0\rightarrow A^0 \rightarrow L \stackrel{\mathbf{b}}{\rightarrow}A\rightarrow 0\] 
For details, see \cite{ChZhSh11}. 

This result, our results from Section \ref{sec:Dorfmanlifts} and the
results from \cite{Jotz18a} as outlined in Remark \ref{sigma_Delta}, suggest
the following relationships between subspaces of
$\Gamma(\widehat{E})\cong\Gamma(\mathcal{E})$ that are closed under
$\lb\cdot,\cdot\rb$ and project to locally-free subsheaves of
$\Gamma(TM\oplus E^*)$, and $\R$-bilinear brackets on subbundles of
$TM\oplus E^*$:

Let $\mathcal{V}\subset \Gamma(\widehat{E})$ be a subspace that is
closed under $\lb\cdot,\cdot\rb$ and such that $\mathcal{V}$ maps to
$\Gamma(F), F$ a subbundle of $TM\oplus E^*$. Then, collectively, we
have the following results:
\subsubsection{Setting 1:} $F=TM\oplus E^*, \mathcal{V}=\Im(\Xi)$, where $\Xi\colon\Gamma(TM\oplus E^*)\rightarrow \Gamma(\widehat{E})$ is a splitting of $p:\Gamma(\widehat{E})\rightarrow \Gamma(TM\oplus E^*)$. This is just the setting of Proposition \ref{rough_lift}, i.e.~such lifts precisely correspond to brackets on $TM\oplus E^*$ that satisfy a Leibniz identity in the second component. \\
Now, if $\mathcal{V}$ is additionally a sub-vector bundle of $\widehat{E}$ and $\Xi$ a morphism of vector bundles, we are in the setting of dull brackets and Dorfman connections, as studied in \cite{Jotz18a}, i.e. the resulting bracket satisfies the Leibniz identity also in its first component.  \\
If instead (or additionally) the lift $\Xi$ is \emph{natural},
i.e.~$\lb\Xi\cdot,\Xi\cdot\rb=\Xi\lb\cdot,\cdot\rb$, the bracket on
$TM\oplus E^*$ satisfies the Dorfman condition (the Jacobi identity in
Leibniz form). \\
If $\mathcal{V}$ is such that
$\left<\nu,\nu'\right>=0$ for all $\nu,\nu'\in \mathcal{V}$, the
bracket $\lb\cdot,\cdot\rb$ on $TM\oplus E^*$ is antisymmetric (see
Theorem \ref{prop_2_chi}).
\subsubsection{Setting 2:} $\mathcal{V}=\Gamma(L), L\subset \widehat{E}$ a Dirac structure. This is the case studied by \cite{ChLi10,ChZhSh11} as described above. The parallels to the first setting are obvious: $\mathcal{V}$ is closed under $\lb\cdot,\cdot\rb$, which is necessary to induce an $\R$-bilinear bracket on its projection to $TM\oplus E^*$ at all, $\mathcal{V}$ is isotropic under $\left<\cdot,\cdot\right>$, so the resulting bracket is antisymmetric, and $\mathcal{V}$ is given by the sections of a vector bundle, i.e. the Leibniz rule in the first component is satisfied. \\
However, in this case there is not necessarily a splitting
$\Xi\colon F\rightarrow L$.
\subsubsection{Setting 3:} Of course the first two settings are not
mutually exclusive: According to our results, Dirac structures
$L\subset \widehat{E}$ which project surjectively to $TM\oplus E^*$
allow a lift $\Xi: TM\oplus E^*\rightarrow L$, which is natural -- a projective Lie
bracket on $TM\oplus E^*$ is in particular a Dorfman bracket.

\section{Standard examples}\label{sec:examples}
We illustrate the result in Theorem \ref{super} with the examples of
standard Dorfman brackets on $TM\oplus E^*$, by giving explicitly the
lifts.
\subsection{Lift of the Courant-Dorfman bracket}\label{lift_CD}
We consider here the case where $E=TM$ and the Dorfman bracket on 
$TM\oplus T^*M$ is the Courant-Dorfman bracket
\[ \lb (X_1,\theta_1), (X_2,\theta_2)\rb=([X_1,X_2], \ldr{X_1}\theta_2-\ip{X_2}\dr\theta_1)
\]
for $X_1,X_2\in\mx(M)$ and $\theta_1,\theta_2\in\Omega^1(M)$.  First,
recall that the derivation $\mathcal D\colon \Gamma(TM\oplus
T^*M)\times\Gamma(TM\oplus T^*M)\to \Gamma(TM\oplus T^*M)$ is 
just 
$\mathcal D_{\nu_1}\nu_2=\lb \nu_1,\nu_2\rb$.
Hence, by definition, the value $\Xi(X,\theta)(Y(m))$
  is
\[\left(T_mYX(m)-\left.\frac{d}{dt}\right\an{t=0}(Y+t[X,Y])(m), 
\dr\ell_\theta-(T_{Y(m)}p_M)^*(-\ip{Y}\dr\theta)\right).
\]
Using \eqref{explicit_hat_D}, we get $\Xi\colon \Gamma(TM\oplus
T^*M)\to\Gamma^l_{TM}(TTM\oplus T^*TM)$,
\begin{equation}\label{lift_of_CD}
\Xi(X,\theta)=\left(\widehat{[X,\cdot]}, \dr\ell_\theta-\widetilde{\dr\theta}\right),
\end{equation}
where $\widetilde{\dr\theta}$ is the one-form on $TM$ defined by
$\widetilde{\dr\theta}(v)=(T_{v}p_M)^*(-\ip{v}\dr\theta)\in T^*_v(TM)$
for all $v\in TM$.  This choice of sign is for consistency with the
notations in the next section for the general case $E$, e.g.~in the
proof of \eqref{inner1}.  We have indeed $\langle
\widetilde{\dr\theta}, \widehat{D}\rangle=\ell_{\ip{X}\dr\theta}$ for
any derivation $D$ of $TM$ over $X\in\mx(M)$, since evaluated at
$Y(m)\in TM$, $\langle \widetilde{\dr\theta},
\widehat{D}\rangle(Y(m))$ is
$\langle(T_{Y(m)}p_M)^*(-\ip{Y(m)}\dr\theta),T_mY(X(m))\rangle=\langle-\ip{Y(m)}\dr\theta,
X(m)\rangle =\ell_{\ip{X}\dr\theta}(Y(m))$.
 
\medskip

For the convenience of the reader, let us compute here explicitly the
Courant-Dorfman bracket $\left\lb \Xi(X_1,\theta_1),
  \Xi(X_2,\theta_2)\right\rb$ of two images of $\Xi$. The Lie bracket
of two linear vector fields $\widehat{D_1},\widehat{D_2}\in\mx^l(E)$
is
$[\widehat{D_1},\widehat{D_2}]=\widehat{[D_1,D_2]}=\reallywidehat{D_1\circ
  D_2-D_2\circ D_1}$. To see this, one only needs to apply
$[\widehat{D_1},\widehat{D_2}]$ on linear and pullback functions.
Since $[[X_1,\cdot], [X_2,\cdot]]$ is $[[X_1,X_2],\cdot]$ by the
Jacobi identity, we find that the Lie bracket of the vector fields
$\widehat{[X_1,\cdot]}$ and $\widehat{[X_2,\cdot]}$ is
$\reallywidehat{[[X_1,X_2],\cdot]}$.  Let us compute
$\ldr{\widehat{[X_1,\cdot]}}(\dr\ell_{\theta_2}-\widetilde{\dr\theta_2})-\ip{\widehat{[X_2,\cdot]}}
\dr(\dr\ell_{\theta_1}-\widetilde{\dr\theta_1})$.  We have $
\ldr{\widehat{[X_1,\cdot]}}\dr\ell_{\theta_2}=\dr\left(\widehat{[X_1,\cdot]}(\ell_{\theta_2})\right)=
\dr\ell_{\ldr{X_1}\theta_2}$ and
\begin{equation*}
\begin{split}
\ldr{\widehat{[X_1,\cdot]}}(\widetilde{\dr\theta_2})=\reallywidetilde{\dr(\ldr{X_1}\theta_2)}.
\end{split}
\end{equation*}
The second equation is more difficult to see and requires some
explanations.  Take $Y\in\mx(M)$. Then 
\begin{equation*}
\begin{split}
\left\langle\ldr{\widehat{[X,\cdot]}}\widetilde{\dr\theta}, \widehat{[Y,\cdot]}\right\rangle
&=\widehat{[X,\cdot]}\left\langle \widetilde{\dr\theta},\widehat{[Y,\cdot]}\right\rangle
- \left\langle \widetilde{\dr\theta}, \reallywidehat{[[X,Y],\cdot]}\right\rangle\\
&=\widehat{[X,\cdot]}\ell_{\ip{Y}\dr\theta}-\ell_{\ip{[X,Y]}\dr\theta}=\ell_{\ldr{X}\ip{Y}\dr\theta-\ip{[X,Y]}\dr\theta}\\
&= \ell_{\ip{Y}\dr\ldr{X}\theta}=\left\langle\widetilde{\dr\ldr{X}\theta}, \widehat{[Y,\cdot]}\right\rangle.
\end{split}
\end{equation*}
Since $\left\langle\ldr{\widehat{[X,\cdot]}}\widetilde{\dr\theta},
  Y^\uparrow\right\rangle$ is easily seen to vanish, as
$\left\langle\widetilde{\dr\ldr{X}\theta},
  Y^\uparrow\right\rangle$ does, we find that
$\ldr{\widehat{[X,\cdot]}}\widetilde{\dr\theta}=\widetilde{\dr\ldr{X}\theta}$.
Therefore we get
\begin{equation*}
\begin{split}
&\ldr{\widehat{[X_1,\cdot]}}(\dr\ell_{\theta_2}-\widetilde{\dr\theta_2})-\ip{\widehat{[X_2,\cdot]}}
\dr(\dr\ell_{\theta_1}-\widetilde{\dr\theta_1})\\
=&\ldr{\widehat{[X_1,\cdot]}}(\dr\ell_{\theta_2}-\widetilde{\dr\theta_2})+\ldr{\widehat{[X_2,\cdot]}}
\widetilde{\dr\theta_1}-\dr(\ip{\widehat{[X_2,\cdot]}}\widetilde{\dr\theta_1})\\
=&\dr\ell_{\ldr{X_1}\theta_2}-\widetilde{\dr(\ldr{X_1}\theta_2)}+\widetilde{\dr(\ldr{X_2}\theta_1)}
-\dr\langle\widetilde{\dr\theta_1}, \widehat{[X_2,\cdot]}\rangle\\
=&\dr\ell_{\ldr{X_1}\theta_2}-\widetilde{\dr(\ldr{X_1}\theta_2)}+\widetilde{\dr(\ldr{X_2}\theta_1)}
-\dr\ell_{\ip{X_2}\dr\theta_1}.
\end{split}
\end{equation*} Since
$\dr(\ldr{X_2}\theta_1)=\dr(\ip{X_2}\dr\theta_1)$, this shows
\begin{equation*}
\begin{split}
  \left\lb \Xi(X_1,\theta_1),
    \Xi(X_2,\theta_2)\right\rb&=\left(\reallywidehat{[[X_1,X_2],\cdot]},
    \dr\ell_{\ldr{X_1}\theta_2-\ip{X_2}\dr\theta_1}- \reallywidetilde{\dr(\ldr{X_1}\theta_2-\ip{X_2}\dr\theta_1)}\right)\\
&=\Xi\lb (X_1,\theta_1), (X_2,\theta_2)\rb.
\end{split}
\end{equation*} 

\begin{remark} 
There is a canonical isomorphism of double vector bundles 
\[ \Sigma: T(TM\oplus T^*M) \rightarrow TTM\oplus T^*TM, \]
which maps the natural VB-Courant algebroid structure on
$T(TM\oplus T^*M)$, the tangent prolongation of the standard Courant
algebroid on $TM\oplus T^*M$, to the standard VB-Courant algebroid
structure on $T(TM)\oplus T^*(TM)$. The lift $\Xi$ is then precisely
$\Xi=\Sigma \circ T$,
where $T$ denotes the tangent prolongation of a section, 
\[ (s\colon M\rightarrow TM\oplus T^*M)\mapsto(Ts:TM\rightarrow T(TM\oplus T^*M)). \]  
A precise description and proof can be found in \cite{JoStXu16}. 
\end{remark} 

\subsection{Another lift to $TTM\oplus T^*TM$}
Consider this time the natural lift $\Xi\colon\Gamma(TM\oplus
T^*M)\to\Gamma_{TM}^l(T(TM)\oplus T^*(TM))$,
$\Xi(X,\theta)=\left(\widehat{[X,\cdot]},\dr\ell_\theta\right)$.  This
is equivalent to the Dorfman bracket
$\lb\cdot\,,\cdot\rb\colon\Gamma(TM\oplus T^*M)\times\Gamma(TM\oplus
T^*M)\to \Gamma(TM\oplus T^*M)$,
\[ \lb (X_1,\theta_1), (X_2, \theta_2)\rb=([X_1,X_2],\ldr{X_1}\theta_2).
\]

To see this, let us compute the Courant-Dorfman bracket of
$\Xi(X_1,\theta_1)$ with $\Xi(X_2,\theta_2)$.  We have
\begin{equation*}
  \left\lb \Xi(X_1,\theta_1), \Xi(X_2,\theta_2)\right\rb=\left(\left[\widehat{[X_1,\cdot]}, \widehat{[X_2,\cdot]}\right],
    \ldr{\widehat{[X_1,\cdot]}}\dr\ell_{\theta_2}-\ip{\widehat{[X_2,\cdot]}}\dr^2\ell_{\theta_1}\right).
\end{equation*}
By the formulas found in the preceding example, we get
\begin{equation}
\left\lb \Xi(X_1,\theta_1), \Xi(X_2,\theta_2)\right\rb=\left(\reallywidehat{[[X_1,X_2],\cdot]},\dr\ell_{\ldr{X_1}\theta_2}\right)
=\Xi([X_1,X_2], \ldr{X_1}\theta_2).
\end{equation}
In fact, we call the lifts associated to Dorfman brackets \enquote{natural} because
they generalise the properties of this lift.

\subsection{More general examples}
More generally, according to \eqref{general_lift} the lift
corresponding to the Dorfman bracket in Example \ref{ex:forms} has the
same form:
\begin{align} 
  \Xi((X,\alpha))(e_m) &=\left( (T_m e)(X(m)) - (\ldr{X} e)^{\uparrow}(e_m), \dr \ell_{\alpha} (e_m) -  \widetilde{\dr \alpha} \right) \nonumber \\
  &= \left(\widehat{\ldr{X} \cdot}, \dr \ell_{\alpha}
    - \widetilde{\dr \alpha} \right)(e_m)
\end{align} 
for all $e_m \in \wedge^k TM$, where, in the second equality, we have
used the definition of the derivation $\widehat{D}$ in
\eqref{explicit_hat_D}.  Here in order to be consistent with the next
section, as well as the previous example, $\widetilde{\dr \alpha}$ is
defined by:
\begin{equation*} 
  \widetilde{\dr \alpha}(e_m)= (T_{e_m} \pr_{TM})^* ((-1)^k \ip{e_m} \dr \alpha) = (T_{e_m} \pr_{TM})^* ( \dr \alpha(\cdot,e_m)).
\end{equation*} 
\medskip

In all examples so far, the lift $\Xi\colon \Gamma(TM \oplus E^*)
\rightarrow \Gamma^l_E (TE \oplus T^*E)$ is really a direct sum of two
lifts $ \Xi_{TM} \colon \mx(M) \rightarrow \Gamma^l_E(TE)$ and
$\Xi_{E^*}\colon \Gamma(E^*) \rightarrow \Gamma^l_E(T^*E)$.  All the
examples discussed so far are split Dorfman brackets. For these, we
always have:
\begin{proposition} 
For all split Dorfman brackets on $TM\oplus E^*$, $\Xi(X,0) \in \mx^l(E)$. 
\end{proposition} 
\begin{proof} 
  We show that $\mathcal{D}_{(X,0)}(e,0)=(\delta_{(X,0)}e, 0)$:
\begin{align*} 
  \left\langle\mathcal{D}_{(X,0)}(e,0), (Y,0) \right\rangle &= X\left\langle(Y,0),(e,0)\right\rangle - \left\langle\lb (X,0),(Y,0) \rb, (e,0) \right\rangle \\
  &= - \left\langle ([X,Y],0),(e,0) \right\rangle = 0
\end{align*} 
for all $Y\in\mx(M)$.
\end{proof} 

However, for general split Dorfman brackets $\Xi_{E^*}$ is a map
$\Gamma(E^*) \rightarrow \Gamma^l_e(TE\oplus T^*E)$. For example the
term in \eqref{equ:mixterm} gives rise to a term $(e \neg \dr
\alpha_k)^{\uparrow}(e_m)\in \Gamma(TE)$ in
$\Xi_{E^*}(\alpha_k)(e_m)$.
If the Dorfman bracket is \emph{not} split, mixing can also occur in
the $TM$-part of the lift: $\Xi_{TM}\colon TM \rightarrow
\Gamma^l_E(TE\oplus T^*E)$, as illustrated by the following example:
\begin{example} 
  Let $H \in \Omega^3_{\rm cl}(M)$ be a closed $3$-form. Then $\lb
  (X_1,\theta_1), (X_2,\theta_2) \rb_H = \lb
  (X_1,\theta_1),(X_2,\theta_2) \rb +(0,\ip{X_2} \ip{X_1}H)$ (with
  $\lb \cdot,\cdot \rb$ the Courant-Dorfman bracket) is also a Dorfman
  bracket on $TM \oplus T^*M$. This Dorfman bracket is not split, and
  we have $\mathcal D^H_{(X,0)}(Y,0)=([X,Y], \ip{Y}\ip{X}H)$ by
  Example \ref{CD_ex}, which shows
\begin{equation*} 
\Xi^H(X,0)(Y(m)) = (\Xi_{TM}(X), -p_M^*(\ip{Y}\ip{X} H))(Y(m)).
\end{equation*} 
The following section studies in detail such \emph{twistings} of
Dorfman brackets in relation to their lifts.
\end{example}

\section{Twisted Courant-Dorfman bracket over vector bundles.} \label{sec:twists}
Here we consider the standard Courant-Dorfman bracket on $TE\oplus
T^*E$ over a vector bundle $E$, twisted by a linear closed 3-form
$H\in\Omega^3(E)$.  That is, we have
\[ \lb (X_1,\alpha_1), (X_2,\alpha_2)\rb_H=\lb (X_1,\alpha_1), (X_2,\alpha_2)\rb +(0, \ip{X_2}\ip{X_1}H).
\]

Given a form $\mu\in\Omega^2(M,E^*)$ and a Dorfman bracket
$\lb\cdot\,,\cdot\rb$ on sections of $TM\oplus E^*$, we can define a
twisted bracket $\lb\cdot\,,\cdot\rb_\mu\colon \Gamma(TM\oplus
E^*)\times\Gamma(TM\oplus E^*)\to \Gamma(TM\oplus E^*)$ by
\[ \lb (X_1,\epsilon_1), (X_2, \epsilon_2)\rb_\mu=\lb (X_1,\epsilon_1), (X_2, \epsilon_2)\rb+(0, \ip{X_2}\ip{X_1}\mu).
\]
This satisfies a Leiniz equality in the second term (as always, with
anchor $\pr_{TM}$) and is compatible with the anchor.  We make the
following definition.
\begin{definition}
  Let $\lb \cdot\,,\cdot\rb\colon\Gamma(TM\oplus E^*)\times
  \Gamma(TM\oplus E^*)\to \Gamma(TM\oplus E^*)$ be a Dorfman bracket
  and $\mu\in\Omega^2(M,E^*)$ a form. Then we say that \textbf{$\mu$
    twists $\lb \cdot\,,\cdot\rb$} if $\lb \cdot\,,\cdot\rb_\mu$
  satisfies the Jacobi identity in Leibniz form, i.e.~if $\lb
  \cdot\,,\cdot\rb_\mu$ is a new Dorfman bracket.
\end{definition}
In this section we will describe in terms of the lift associated to
$\lb\cdot\,,\cdot\rb$ a necessary and sufficient condition for $\mu$
to twist $\lb \cdot\,,\cdot\rb$.  

\begin{example} 
  The standard Dorfman bracket on $TM \oplus \wedge^k T^*M$ (Example
  \ref{ex:forms}) is twisted by $\mu \in \Omega^2(M, \wedge^k T^*M)$
  if and only if $\mu \in \Omega^{k+2}_{\rm cl}(M)$, i.e. actually
  antisymmetric in all components and closed.
\end{example} 

We define the dual derivation
$\mathcal D^\mu\colon \Gamma(TM\oplus E^*)\times\Gamma(E\oplus T^*M)\to \Gamma(E\oplus T^*M)$
to $\lb \cdot\,,\cdot\rb_\mu$
and find 
\begin{equation}\label{twist_D}
\mathcal D^\mu_{(X,\epsilon)}(e,\theta)=\mathcal D_{(X,\epsilon)}(e,\theta)-(0, \langle\ip{X}\mu, e\rangle).
\end{equation}
The corresponding lift $\Xi^\mu\colon \Gamma(TM\oplus E^*)\to \Gamma^l_E(TE\oplus T^*E)$ as in \eqref{general_lift}
is then just
\[\Xi^\mu(X,\epsilon)=\Xi(X,\epsilon)+\widetilde{(0, \ip{X}\mu)}.
\]
Recall that it is natural if and only if $\lb \cdot\,,\cdot\rb_\mu$
satisfies the Jacobi identity.

\begin{theorem}\label{twist1}
With the notations above, we have
\[ \left\lb\Xi^\mu(\nu_1), \Xi^\mu(\nu_2)\right\rb_{-\dr\Lambda_\mu}=\Xi^\mu\lb\nu_1,\nu_2\rb
\]
for all $\nu_1,\nu_2\in\Gamma(TM\oplus E^*)$.
\end{theorem}

\begin{proof}
We just compute
\begin{align*}
&\left\lb \Xi^\mu(\nu_1), \Xi^\mu(\nu_2)\right\rb_{-\dr\Lambda_\mu}
=\left\lb \Xi(\nu_1)+ \widetilde{(0, \ip{X_1}\mu)}, \Xi(\nu_2)+\widetilde{(0, \ip{X_2}\mu)}\right\rb_{-\dr\Lambda_\mu}\\
\overset{\eqref{inner1}}{=}&\left\lb \Xi(\nu_1), \Xi(\nu_2)\right\rb-\left(0, \dr\ell_{\ip{X_2}\ip{X_1}\mu}+\widetilde{\mathcal D_{\nu_1}(\ip{X_2}\mu)}-\widetilde{\mathcal D_{\nu_2}(\ip{X_1}\mu)}-\widetilde{\ip{[X_1,X_2]}\mu}\right)\\
&+\left\lb \Xi(\nu_1), \widetilde{(0, \ip{X_2}\mu)}\right\rb
+\left\lb\widetilde{(0, \ip{X_1}\mu)}, \Xi(\nu_2)\right\rb\\
=\,&\,\Xi\lb\nu_1, \nu_2\rb-\left(0, \dr\ell_{\ip{X_2}\ip{X_1}\mu}+\widetilde{\mathcal D_{\nu_1}(\ip{X_2}\mu)}-\widetilde{\mathcal D_{\nu_2}(\ip{X_1}\mu)}-\widetilde{\ip{[X_1,X_2]}\mu}\right)\\
&+\left(0, \widetilde{\mathcal D_{\nu_1}(\ip{X_2}\mu)}\right)+\left(0, \dr\ell_{\ip{X_2}\ip{X_1}\mu}\right)
-\left(0, \widetilde{\mathcal D_{\nu_2}(\ip{X_1}\mu)}\right)\\
=\,&\,\Xi^\mu\lb\nu_1, \nu_2\rb\qedhere
\end{align*}
\end{proof}

In the third equality, we have used 
Lemma \ref{new_lemma}. 
We are now ready to prove our main theorem.
\begin{theorem} \label{twist2}
Consider a Dorfman bracket
\[\lb \cdot\,,\cdot\rb\colon\Gamma(TM\oplus E^*)\times \Gamma(TM\oplus E^*)\to \Gamma(TM\oplus E^*)
\]
and the corresponding lift $\Xi\colon\Gamma(TM\oplus T^*M)\to \Gamma^l_{TM}(TE\oplus T^*E)$.

Then a form $\mu\in\Omega^2(M,E^*)$ twists $\lb \cdot\,,\cdot\rb$  if and only if
\begin{equation*}
  \left\lb \Xi(\nu_1), \Xi(\nu_2)\right\rb_{\dr\Lambda_\mu}=\Xi\lb \nu_1,\nu_2\rb_\mu
\end{equation*}
for all $\nu_1, \nu_2\in\Gamma(TM\oplus TM^*)$.
\end{theorem}
In other words, \emph{$\mu$ twists a Dorfman bracket if and only its
  natural lift lifts the twisted bracket to the twist by
  $\dr\Lambda_\mu$ of the Courant-Dorfman bracket.}

Note that we also
have the following result, which follows from \eqref{inner2} and \eqref{twist_D}. 
\begin{proposition}\label{prop_linear_core_twist}
  In the situation of the previous theorem, we have
\begin{equation*}
\left\lb\Xi(\nu), \tau^\uparrow\right\rb_{\dr\Lambda_\mu}=\mathcal D_\nu^\mu\tau^\uparrow,
\end{equation*}
for $\nu\in\Gamma(TM\oplus E^*)$ and $\tau\in\Gamma(E\oplus T^*M)$,
no matter if $\mu$ twists the Dorfman bracket or not.
\end{proposition}

\begin{proof}[Proof of Theorem \ref{twist2}]
Assume that $\lb \cdot\,,\cdot\rb_\mu$ is a Dorfman bracket. Then by Theorem \ref{super}, we have 
\begin{equation}\label{eq1}
 \lb \Xi^\mu(\nu_1), \Xi^\mu(\nu_2)\rb=\Xi^\mu\lb \nu_1, \nu_2\rb_\mu
=\Xi^\mu\lb \nu_1, \nu_2\rb+\Xi^\mu(0,\mu(X_1,X_2)).
\end{equation}
Since $\Xi^\mu(\nu)=\Xi(\nu)+\widetilde{(0,\ip{X}\mu)}$, we find that
\begin{equation}\label{eq1,5}\Xi^\mu(0,\mu(X_1,X_2))=\Xi(0,\mu(X_1,X_2))
\end{equation} and also that
$\pr_{TE}\Xi^\mu(\nu)=\pr_{TE}\Xi(\nu)=\widehat{\delta_\nu}$ for all
$\nu\in\Gamma(TM\oplus E^*)$.  By Theorem \ref{twist1}, we have
\begin{equation}\label{eq2}
\lb \Xi^\mu(\nu_1), \Xi^\mu(\nu_2)\rb=\Xi^\mu\lb \nu_1,\nu_2\rb
+\left(0, \ip{\widehat{\delta_{\nu_2}}}\ip{\widehat{\delta_{\nu_1}}}\dr\Lambda_\mu\right).
\end{equation}
\eqref{eq1}, \eqref{eq1,5} and \eqref{eq2} yield together
$\Xi(0,\mu(X_1,X_2))=\left(0, \ip{\widehat{\delta_{\nu_2}}}\ip{\widehat{\delta_{\nu_1}}}\dr\Lambda_\mu\right)$,
and so 
\begin{equation*}
\begin{split}
  \left\lb \Xi(\nu_1), \Xi(\nu_2)\right\rb_{\dr\Lambda_\mu}
  &=\left\lb \Xi(\nu_1), \Xi(\nu_2)\right\rb+\left(0, \ip{\widehat{\delta_{\nu_2}}}\ip{\widehat{\delta_{\nu_1}}}\dr\Lambda_\mu\right)\\
  &=\Xi\lb \nu_1, \nu_2\rb+\Xi(0,\mu(X_1,X_2))=\Xi\lb
  \nu_1,\nu_2\rb_\mu.\qedhere
\end{split}
\end{equation*}
\end{proof}

\begin{example}
  Consider  $E=TM$ and choose the Courant-Dorfman bracket
  on $TM\oplus T^*M$. Recall from \S\ref{lift_CD} the corresponding
  natural lift.  Then if $\nu_1=(X_1,\theta_1)$, we get $\delta_\nu
  X_2=[X_1,X_2]$ and
  $\mathcal D_{\nu_1}(\ip{X_2}\mu)=\ldr{X_1}\ip{X_2}\mu$. As a
  consequence,
\begin{equation}
\begin{split}
\mathcal D_{\nu_1}(\ip{X_2}\mu)-\mathcal D_{\nu_2}(\ip{X_1}\mu)-\ip{[X_1,X_2]}\mu&=
\ip{X_2}\ldr{X_1}\mu-\ldr{X_2}\ip{X_1}\mu\\
&=\ip{X_2}\ip{X_1}\dr\mu-\dr(\ip{X_2}\ip{X_1}\mu)
\end{split}
\end{equation}
and 
$\mathcal D_{\nu_1}(\ip{X_2}\mu)-\mathcal D_{\nu_2}(\ip{X_1}\mu)-\ip{[X_1,X_2]}\mu=-\dr(\ip{X_2}\ip{X_1}\mu)$
if and only if $\mu$ is closed.
We get then using \eqref{inner1}
\begin{equation*}
\begin{split}
 & \left\lb\Xi(X_1,\theta_1),\Xi(X_2,\theta_2)\right\rb_{\dr\Lambda_\mu}\\
  &=\Xi\lb(X_1,\theta_1),(X_2,\theta_2)\rb+\left(0,\ip{\widehat{[X_2,\cdot]}}\ip{\widehat{[X_1,\cdot]}}\dr\Lambda_\mu\right)\\
&=\Xi\lb(X_1,\theta_1),(X_2,\theta_2)\rb+\left(0,\dr\ell_{\ip{X_2}\ip{X_1}\mu}
-\widetilde{\dr(\ip{X_2}\ip{X_1}\mu)}\right)=\Xi\lb(X_1,\theta_1),(X_2,\theta_2)\rb_\mu.
\end{split}
\end{equation*}
\end{example}

\section{Symmetries of Dorfman brackets} \label{sec:symm} In this
section we use the known symmetries of the standard Courant algebroid
over $E$ to study a similar class of symmetries for Dorfman brackets
on $TM \oplus E^*$.

Consider $B\in \Omega^2_{\rm cl}(E)$.
We denote by $\Phi_B\colon
TE\oplus T^*E\to TE\oplus T^*E$ the vector bundle morphism over the identity on $E$ that is defined by
\[\Phi_B(X,\theta)=(X,\theta+\ip{X}B)
\]
for all $X\in\mx(E)$ and $\theta\in\Omega^1(E)$.
Then $\Phi_B$ is a symmetry of the Courant-Dorfman bracket on $TE
\oplus T^*E$ \cite{BuCaGu07}:
\[\lb \Phi_B(\chi_1), \Phi_B(\chi_2)\rb
=\Phi_B\lb\chi_1, \chi_2\rb
\]
for all $\chi_1,\chi_2\in\Gamma(TE\oplus T^*E)$.

According to \cite{BuCa12} (see Section \ref{sec:twists}), given a
form $\beta\in\Omega^1(M,E^*)$, the closed form $B=-\dr\Lambda_{\beta}$ is
linear.  In particular, if $\lb\cdot\,,\cdot\rb$ is a Dorfman bracket
on $TM \oplus E^*$ and $\Xi\colon \Gamma(TM\oplus E^*) \rightarrow
\Gamma_E^{l}(TE\oplus T^*E)$ the associated lift,
$\Phi_B(\Xi(\nu))=\Xi(\nu)+\ip{\widehat{\delta_\nu}}B$ is a linear section
of $TE\oplus T^*E$ over $\Phi_{\beta}(\nu)=\nu+(0,\ip{X}\beta)$
(see Lemma \ref{inner_product_linear}), where
 $\Phi_{\beta}\colon
TM\oplus E^*\to TM\oplus E^*$ is the vector bundle morphism over the identity on $M$:
\[\Phi_{\beta}(X,\epsilon)=(X,\epsilon+\ip{X}\beta).
\]
In this section we aim to understand when this map defines a symmetry of a
Dorfman bracket on $TM\oplus E^*$.
We prove the following result.
\begin{theorem} \label{thm:symm} A form $\beta\in\Omega^1(M,E^*)$
  defines a symmetry of a Dorfman bracket $\lb\cdot,\cdot\rb$ via
  $(X,\epsilon) \mapsto (X,\epsilon+\ip{X}\beta)$ if and only if
\[\Phi_{-\dr\Lambda_{\beta}}\circ\Xi=\Xi\circ\Phi_{\beta}
\]
for the corresponding lift
$\Xi\colon\Gamma(TM \oplus E^*) \rightarrow \Gamma^l_E(TE\oplus T^*E)$.
\end{theorem}

The proof relies on the following lemma. We set $B:=-\dr\Lambda_{\beta}$ for
  $\beta\in\Omega^1(M,E^*)$.
\begin{lemma} \label{lem:linvert} Choose
  $\phi \in \Gamma(\operatorname{Hom}(E,E \oplus T^*M))$.  Then
  $\dr\langle \Phi_B(\Xi(\nu)),\widetilde{\phi} \rangle$ is a
  core linear section of $T^*E\to E$ for all
  $\nu\in \Gamma(TM\oplus E^*)$ if and only if $\phi=0$.
\end{lemma}

\begin{proof} 
Since $\langle\tilde\phi, \Phi_B(\Xi(\nu)) \rangle$ is linear, $\dr\langle \tilde\phi,
\Phi_B(\Xi(\nu)) \rangle$ is a core linear section if and only if
$\Phi_E\left(\dr \left\langle\Phi_B(\Xi(\nu)),\tilde\phi\right\rangle\right)=0$. We have
\begin{equation*} 
\begin{split}
  \left\langle\Phi_E\left( \dr
    \left\langle\Phi_B(\Xi(\nu)),\tilde\phi\right\rangle\right), e
  \right\rangle&= \left\langle
  \dr  \left\langle\Phi_B(\Xi(\nu)),\tilde\phi\right\rangle, e^{\uparrow}
  \right\rangle \\
&=
  e^{\uparrow}\left\langle\Phi_B(\Xi(\nu)),\tilde\phi\right\rangle.
\end{split}
\end{equation*} 
Write $\nu=(X, \epsilon) \in \Gamma(TM\oplus E^*)$.
Since $\Phi_B(\Xi(X,\epsilon))= \Xi(X,\epsilon) +(0, \dr\ell_{\beta(X)}-\widetilde{\delta_{(X,\epsilon)}\beta})
$, we find 
\[\left\langle\Phi_B(\Xi(\nu)),\tilde\phi\right\rangle=\ell_{\phi^*(X,\epsilon+\beta(X))}
\]
and so
$e^{\uparrow}\left\langle\Phi_B(\Xi(\nu)),\tilde\phi\right\rangle=q_E^*\langle \phi^*(X,\epsilon+\beta(X)), e\rangle$. This vanishes for all $e\in\Gamma(E)$ and all
$(X,\epsilon)\in\Gamma(TM\oplus E^*)$ if and only if 
$\phi^*(X,\epsilon+\beta(X))=0$ for all $(X,\epsilon)\in\Gamma(TM\oplus E^*)$. In particular,
$\phi^*(0,\epsilon)$ must be $0$ for all $\epsilon\in\Gamma(E^*)$ or,
in other words, $\phi$ must have image in $T^*M$.  Using this, we find
$\phi^*(X,0)=\phi^*(X,\beta(X))$ for $X\in\mx(M)$. Since this must vanish for all $X\in\mx(M)$,
we have shown that $\phi$ must be $0$.
\end{proof}

\begin{proof}[Proof of Theorem \ref{thm:symm}]
 We define
  $\phi_{(X,\epsilon)}\in\Gamma(\operatorname{Hom}(E,E\oplus T^*M))$
  by
\begin{equation}\label{def_of_difference}
\widetilde{\phi_{(X,\epsilon)}}=\Xi(0,\ip{X}\beta)
-(0,\ip{\widehat{\delta_{(X,\epsilon)}}}B) =\Xi(0,\ip{X}\beta)
-\left(0,\dr\ell_{\beta(X)}-\widetilde{\delta_{(X,\epsilon)}\beta}\right) .
\end{equation}
We have used Lemma \ref{inner_product_linear}.  Note that this
difference is a core-linear section of $TE \oplus T^*E$ because the
linear sections $\Xi(0,\ip{X} \beta)$ and
$\left(0,\dr\ell_{\beta(X)}-\widetilde{\delta_{(X,\epsilon)}\beta}\right)$
both project to $(0,\ip{X}\beta)$ in $\Gamma(TM \oplus E^*)$.

Consider 
\[\lb\Phi_{\beta}(X_1,\epsilon_1), \Phi_{\beta}(X_2, \epsilon_2)\rb
=\lb (X_1, \epsilon_1+\ip{X_1}\beta), (X_2, \epsilon_2+\ip{X_2}\beta)\rb
\] in $\Gamma(TM \oplus E^*)$. This lifts to 
$\Xi\lb (X_1, \epsilon_1+\ip{X_1}\beta), (X_2, \epsilon_2+\ip{X_2}\beta)\rb,
$
which equals
\begin{align*}
  \left\lb\Xi(X_1, \epsilon_1)+\Xi(0,\ip{X_1}\beta), \Xi(X_2, \epsilon_2)+\Xi(0,\ip{X_2}\beta)\right\rb
\end{align*}
But this is 
\begin{align*}
  \left\lb\Phi_B(\Xi(X_1,
    \epsilon_1))+\widetilde{\phi_{(X_1,\epsilon_1)}}, \Phi_B(\Xi(X_2,
    \epsilon_2)))+\widetilde{\phi_{(X_2,\epsilon_2)}}\right\rb,
\end{align*}
which can be expanded to
\begin{align}
 & \Phi_B(\Xi\lb (X_1, \epsilon_1), (X_2, \epsilon_2)\rb) 
+\left\lb \Phi_B(\Xi(X_1, \epsilon_1)),\widetilde{\phi_{(X_2,\epsilon_2)}}\right\rb \nonumber\\
&+\left\lb
      \widetilde{\phi_{(X_1,\epsilon_1)}}, \Phi_B(\Xi(X_2,
      \epsilon_2)))\right\rb +\left\lb
      \widetilde{\phi_{(X_1,\epsilon_1)}},\widetilde{\phi_{(X_2,\epsilon_2)}}\right\rb\label{expand}
\end{align}
The second and fourth terms are again core-linear (see Lemma
\ref{new_lemma} and Lemma 4.5 in \cite{Jotz19b}, respectively) so
project to $0$, but the third is
\[ - \left\lb \Phi_B(\Xi(X_2,
  \epsilon_2)),\widetilde{\phi_{(X_1,\epsilon_1)}}\right\rb 
+\left(0, \dr\langle \Phi_B(\Xi(X_2,
  \epsilon_2)),\widetilde{\phi_{(X_1,\epsilon_1)}}
\rangle\right).
\]
The left-hand term is core-linear, so projects to $0$. 
By Lemma \ref{lem:linvert},
the right-hand term also has values in the core for arbitrary
$(X_2,\epsilon_2) \in \Gamma(TM\oplus E^*)$ if and only if 
$\phi_{(X_1,\epsilon_1)}=0$.
This happens exactly when \eqref{expand} projects to
$\lb(X_1,\epsilon_1),(X_2,\epsilon_2)\rb + (0,\ip{[X_1,X_2]}\beta)$ on $TM\oplus T^*M$, so when
\[ \lb(X_1,\epsilon_1+\ip{X_1} \beta), (X_2,\epsilon_2+\ip{X_2}
\beta)\rb=\lb(X_1,\epsilon_1),(X_2,\epsilon_2)\rb +(0,
\ip{[X_1,X_2]}\beta).\]

Now 
$\widetilde\phi_{(X,\epsilon)}=0$ is equivalent to $(\Xi\circ\Phi_{\beta})(X,\epsilon)=(\Phi_B\circ\Xi)
  (X,\epsilon)$ because 
\begin{equation*}
\begin{split}
  (\Xi\circ\Phi_{\beta})(X,\epsilon)&=\Xi(X,\epsilon+\beta(X))=\Xi(X,\epsilon)+\Xi(0,\ip{X}\beta)\\
  &=\Xi(X,\epsilon)+(0,\ip{\widehat{\delta_{(X,\epsilon)}}}B)+\widetilde\phi_{(X,\epsilon)}=(\Phi_B\circ\Xi)
  (X,\epsilon)+\widetilde\phi_{(X,\epsilon)}. \qedhere
\end{split}
\end{equation*}
\end{proof} 

Note that so far, we have not made any statement as to the existence
of forms like in Theorem \ref{thm:symm}. The theorem rather provides a
simple reformulation of the condition for being a symmetry.
\begin{example} 
  Consider $E=\wedge^k TM$ and the standard Dorfman bracket on
  $TM\oplus \wedge^k T^*M$ already studied earlier. Choose a morphism $\beta\colon TM
  \rightarrow \wedge^k T^*M$ and consider $-\dr \Lambda_\mu$ the associated
  linear $2$-form on $E=\wedge^k TM$.

For $\beta$ to define a symmetry of the Dorfman bracket on $TM\oplus \wedge^k T^*M$, we need 
\begin{align*} 
  \Xi(0,\ip{X}\beta)(e_m) &= -\dr \Lambda_\beta
  (\widehat{\delta_{(X,\alpha_k)}})(e_m) \end{align*} for all $e_m \in
E$, which is equivalent to $(0, \dr \ell_{\ip{X}\beta} - \widetilde{\dr
  \ip{X} \beta}) =- \dr \Lambda_\beta (\widehat{\ldr{X}})$.

Both sides of this equation are sections of $T^*E$, and they are equal
if and only if they map all linear and all core vector fields in the
same way.  On core vector fields $T^{\uparrow}$, for $T\in
\Gamma(\wedge^k TM)$, we have
\begin{align*} 
  \dr \Lambda_\beta (\widehat{\ldr{X} }, T^{\uparrow}) &= \widehat{\ldr{X} }(\Lambda_\beta (T^{\uparrow})) - T^{\uparrow}(\Lambda_\beta(\widehat{\ldr{X} })) - \Lambda_\beta([\widehat{\ldr{X} } , T^{\uparrow}]) \\
&= 0-q_E^*\langle T, \ip{X}\beta\rangle-0=0,
\end{align*}
$\dr \ell_{\ip{X}\beta}(T^{\uparrow})= q_E^*\langle T, \ip{X}\beta\rangle$ and 
$\widetilde{\dr \ip{X}\beta}(T^{\uparrow})(e_m) = 0$.
On a linear vector field $\widehat D\in \mx^l(E)$ over $Y\in\mx(M)$, we have  
\begin{align*} 
\dr \Lambda_\beta (\widehat{\ldr{X} }, \widehat{D}) &= \widehat{\ldr{X} }(\Lambda_\beta (\widehat{D})) - \widehat{D}(\Lambda_\beta(\widehat{\ldr{X} })) - \Lambda_\beta([\widehat{\ldr{X} } , \widehat{D}]) \\ 
&=\ell_{\ldr{X}(\ip{Y}\beta)-D^*(\ip{X}\beta)-\ip{[X,Y]}\beta},
\end{align*}
$\widetilde{\dr\ip{X}\beta}(\widehat{D})=\ell_{\ip{Y}\dr \ip{X}\beta}$ and
$\widehat{D }(\ell_{\ip{X}\beta})  = \ell_{D^*(\ip{X}\beta)} $.
Thus we are left with the following condition on $\beta$: 
\begin{align*} 
{\ldr{X}(\ip{Y}\beta)} - \ip{[X,Y]} \beta - D^*(\ip{X}\beta) =  \ip{Y} \dr \ip{X} \beta-  D^*(\ip{X}\beta)
\end{align*} for all $X,Y \in \mx(M) $, which is equivalent to
 $\beta\in\Omega^{k+1}(M)$ and $\dr\beta=0$.
\end{example}

\appendix

\section{On the proofs of Theorems \ref{prop_1_chi} and \ref{prop_2_chi}}\label{proofs}
Choose a linear section $\chi$ of $TE\oplus T^*E\to E$ over a pair
$(X,\varepsilon)\in\Gamma(TM\oplus E^*)$.  Then
$\chi=\left(\widehat{d_\chi},
  \dr\ell_\varepsilon-\widetilde{\phi_\chi}\right)$, following the
notations set after Theorem \ref{linear_are_der_thm}. For simplicity,
we write $\theta_\chi$ for
$\dr\ell_\varepsilon-\widetilde{\phi_\chi}\in\Omega^1(E)$.

\begin{lemma}\label{technical_one}
  Choose linear sections $\chi,\chi'$ of $TE\oplus T^*E\to E$ over
  $(X,\varepsilon),(X',\varepsilon')\in\Gamma(TM\oplus E^*)$, a
  section $e\in\Gamma(E)$ and a derivation $D$ of $E$ with symbol
  $Y$. Then
\begin{enumerate}
\item $\langle\theta_{\chi}, e^\uparrow\rangle=q_E^*\langle
  \varepsilon, e\rangle$,
\item $\langle\theta_{\chi}, \widehat D\rangle
=\ell_{D^*\varepsilon-\phi_\chi^*(Y)}$,
\item $\ldr{e^\uparrow}\theta_{\chi}=q_E^*\left(\dr\langle
    \varepsilon,e\rangle-\phi_\chi(e)\right)$,
\item $\ldr{\widehat{d_{\chi'}}}\theta_{\chi}=\dr\ell_{d_{\chi'}^*\varepsilon}
-\widetilde{(d_{\chi'}(\phi_\chi^*))^*}
$.
\end{enumerate}
\end{lemma}
Note that in the last equation, $\phi_\chi^*$ is an element of
$\Omega^1(M,E^*)$. For a derivation $D$ of $E$ over $X\in\mx(M)$, the
derivation $D\colon\Omega^1(M,E^*)\to\Omega^1(M,E^*)$ over $X$ is
defined by
$(D\omega)(Y)=D^*(\omega(Y))-\omega[X,Y]$
for all $Y\in\mx(M)$.

\begin{proof}
 The first identity is
immediate. For the second, we recall \eqref{explicit_hat_D}. The
pairing of $\hat D$ with $\theta_\chi$ at $e_m$ is 
\[Y\langle \varepsilon, e\rangle- \langle \phi_\chi(e), Y\rangle-\langle\varepsilon, De\rangle
=\langle D^*\varepsilon, e\rangle- \langle \phi_\chi(e), Y\rangle
\] at $m$. 
Hence we have found (2). Next we prove (3). We have 
\begin{equation*}
\begin{split}
\langle \ldr{e^\uparrow}\theta_{\chi}, {e'}^\uparrow\rangle&=e^\uparrow\langle\theta_{\chi}, {e'}^\uparrow\rangle
-\langle \theta_{\chi}, \left[e^\uparrow, {e'}^\uparrow\right]\rangle=e^\uparrow(q_E^*\langle \varepsilon, e'\rangle)=0
\end{split}
\end{equation*}
for $e'\in\Gamma(E)$ and
\begin{equation*}
\begin{split}
  \langle \ldr{e^\uparrow}\theta_\chi,
  \widehat{D}\rangle&=e^\uparrow\langle\theta_\chi,
  \widehat{D}\rangle
  -\langle \theta_\chi, \left[e^\uparrow, \widehat{D}\right]\rangle=q_E^*\langle D^*\varepsilon-\phi_\chi^*(Y),e\rangle
+\langle \theta_\chi, (De)^\uparrow\rangle\\
&=q_E^*\langle  D^*\varepsilon-\phi_\chi^*(Y),e\rangle
+q_E^*\langle\varepsilon, De\rangle=q_E^*(Y\langle\varepsilon, e\rangle-\langle Y, \phi_\chi(e)\rangle)
\end{split}
\end{equation*} for a derivation $D$ of $E$ over $Y\in\mx(M)$.
Since $q_E^*\left(\dr\langle
    \varepsilon,e\rangle-\phi_\chi(e)\right)$ takes the same values
on ${e'}^\uparrow$ and $\hat D$, we are done.
 
Finally, we compute using the first identity
\begin{equation*}
\begin{split}
  \langle \ldr{\widehat{d_{\chi'}}}\theta_\chi,
  e^\uparrow\rangle&=\widehat{d_{\chi'}}\langle\theta_\chi,
  e^\uparrow\rangle
  -\langle \theta_\chi, \left[\widehat{d_{\chi'}}, e^\uparrow\right]\rangle=q_E^*(X'\langle \varepsilon, e\rangle-\langle\varepsilon,d_{\chi'}e\rangle)\\
  &=q_E^*\langle d_{\chi'}^*\varepsilon, e\rangle=\langle
  \dr\ell_{d_{\chi'}^*\varepsilon}, e^\uparrow\rangle=\langle
  \dr\ell_{d_{\chi'}^*\varepsilon}-\widetilde{(d_{\chi'}(\phi_\chi^*))^*},
  e^\uparrow\rangle
\end{split}
\end{equation*}
for $e\in\Gamma(E)$. Similarly, using (2) above
\begin{equation*}
\begin{split}
  \langle
  \ldr{\widehat{d_{\chi'}}}\theta_\chi,
  \widehat{D}\rangle&=\widehat{d_{\chi'}}\langle\theta_\chi,
  \widehat{D}\rangle
  -\langle \theta_\chi, \left[\widehat{d_{\chi'}}, \widehat{D}\right]\rangle=\ell_{d_{\chi'}^*(D^*\varepsilon-\phi_\chi^*(Y))}-\langle \theta_\chi,
  \reallywidehat{[d_{\chi'},D]}\rangle\\
&=\ell_{d_{\chi'}^*(D^*\varepsilon-\phi_\chi^*(Y))
-[d_{\chi'},D]^*\varepsilon+\phi_\chi^*[X',Y]}
\end{split}
\end{equation*} for a derivation $D$ of $E$ over $Y\in\mx(M)$.
An easy calculation shows 
$[d_{\chi'},D]^*=[d_{\chi'}^*,D^*]$,
which leads to 
\begin{equation*}
\begin{split}
\langle
  \ldr{\widehat{d_{\chi'}}}\theta_\chi,
  \widehat{D}\rangle&=\ell_{D^*d_{\chi'}^*\varepsilon-(d_{\chi'}\phi_\chi^*)(Y)}=\langle
  \dr\ell_{d_{\chi'}^*\varepsilon}-\widetilde{(d_{\chi'}\phi_\chi^*)^*},
  \hat D\rangle.\qedhere
\end{split}
\end{equation*}
\end{proof}

\begin{proof}[Proof of Theorem \ref{prop_1_chi}]
  We write
  $\tau=(e,\theta)\in\Gamma(E\oplus T^*M)$.  First we find that
  $\langle\ldr{\widehat{d_\chi}}q_E^*\theta, {e'}^\uparrow\rangle$
equals $\widehat{d_\chi}\langle q_E^*\theta,
  {e'}^\uparrow\rangle-\langle
  q_E^*\theta,[\widehat{d_\chi},{e'}^\uparrow]\rangle =0-0=0$ and
  $\langle\ldr{\widehat{d_\chi}}q_E^*\theta, \widehat D\rangle
  =\widehat{d_\chi}(q_E^*\langle \theta, Y\rangle)-\langle
  q_E^*\theta,[\widehat{d_\chi},\widehat D]\rangle
  =q_E^*(X\langle\theta,
  Y\rangle-\langle\theta,[X,Y]\rangle)=q_E^*\langle\ldr{X}\theta,
  Y\rangle $ for all $e'\in\Gamma(E)$ and any derivation $D$ of $E$
  over $Y\in\mx(M)$. This shows
  $\ldr{\widehat{d_\chi}}q_E^*\theta=q_E^*(\ldr{X}\theta)$.  In
  the same manner, we have
  $\ip{e^\uparrow}\dr\theta_\chi=\ldr{e^\uparrow}\theta_\chi-\dr\langle\theta_\chi,
  e^\uparrow\rangle =q_E^*(-\phi_\chi(e))$
  by (1) and (3) in Lemma \ref{technical_one}.  We get
\begin{equation*}
\begin{split}
  \left\lb \chi, \tau^\uparrow\right\rb
  &=\left(\left[\widehat{d_\chi}, e^\uparrow\right],
    \ldr{\widehat{d_\chi}}q_E^*\theta-\ip{e^\uparrow}\dr\theta_\chi
  \right)=\left( (d_\chi e)^\uparrow,
    q_E^*(\ldr{X}\theta+\pr_{T^*M}D_\chi(e,0))
  \right)\\
  &=\left( (d_\chi e)^\uparrow,
    q_E^*(\pr_{T^*M}{D}_{\chi}(e,\theta)) \right)
=D_\chi\tau^\uparrow,
\end{split}
\end{equation*}
which proves Theorem \ref{prop_1_chi}.
\end{proof}

\begin{proof}[Proof of Theorem \ref{prop_2_chi}]
  We simply compute \begin{equation}\label{bracket}
\begin{split}
&  \lb\chi_1,\chi_2\rb = \left(\left[\widehat{d_{\chi_1}},
      \widehat{d_{\chi_2}}\right],\ldr{\widehat{d_{\chi_1}}}\theta_{\chi_2}-\ip{\widehat{d_{\chi_2}}}\dr\theta_{\chi_1}
      \right).
\end{split}
\end{equation}
The $TE$-part is $\widehat{[d_{\chi_1},d_{\chi_2}]}$. By definition of $D_\chi$, we have 
$\pr_E\circ D_\chi\circ\iota_E\circ\pr_E=\pr_E\circ D_\chi$ and so
$[d_{\chi_1},d_{\chi_2}]=\pr_E\circ [D_{\chi_1},D_{\chi_2}]\circ\iota_E$.

The $T^*E$-component of \eqref{bracket} is 
\begin{equation*}
\dr\ell_{d_{\chi_1}^*\varepsilon_2}
-\reallywidetilde{(d_{\chi_1}(\phi_{\chi_2}^*))^*}
\cancel{-\dr\ell_{d_{\chi_2}^*\varepsilon_1}}
+\reallywidetilde{(d_{\chi_2}(\phi_{\chi_1}^*))^*}
+\dr\ell_{\cancel{d_{\chi_2}^*\varepsilon_1}-\phi_{\chi_1}^*(X_2)}
\end{equation*}
by Lemma \ref{technical_one}. 
First we find that $\langle d_{\chi_1}^*\varepsilon_2-\phi_{\chi_1}^*(X_2), e\rangle$ equals
\begin{equation*}
\begin{split}
&X_1\langle\varepsilon_2,e\rangle-\langle\varepsilon_2,d_{\chi_1}e\rangle
-\langle X_2, \phi_{\chi_1}(e)\rangle\\
&=X_1\langle\varepsilon_2,e\rangle-\langle (X_2,\varepsilon_2), D_{\chi_1}(e,0)\rangle=
\langle D_{\chi_1}^*(X_2,\varepsilon_2), (e,0)\rangle
\end{split}
\end{equation*}
for any $e\in\Gamma(E)$. Then we find that $\langle (d_{\chi_1}\phi_{\chi_2}^*-d_{\chi_2}\phi_{\chi_1}^*)^*(e),X\rangle$ equals
\begin{equation*}
\begin{split}
&\left(d_{\chi_1}^*(\phi_{\chi_2}^*(X))-\phi_{\chi_2}^*[X_1,X]-d_{\chi_2}^*(\phi_{\chi_1}^*(X))+\phi_{\chi_1}^*[X_2,X]
\right)(e)\\
=&X_1\langle X, \phi_{\chi_2}(e)\rangle -\langle X, \phi_{\chi_2}(d_{\chi_1}(e))\rangle -\langle [X_1,X], \phi_{\chi_2}(e)\rangle\\
&-X_2\langle X, \phi_{\chi_1}(e)\rangle +\langle X, \phi_{\chi_1}(d_{\chi_2}(e))\rangle +\langle [X_2,X], \phi_{\chi_1}(e)\rangle\\
=&\langle X, \ldr{X_1}(\pr_{T^*M}D_{\chi_2}(e,0))+\pr_{T^*M}\circ D_{\chi_2}\circ\iota_E\circ\pr_E\circ D_{\chi_1}(e,0)\rangle\\
&-\langle X, \ldr{X_2}(\pr_{T^*M}D_{\chi_1}(e,0))+\pr_{T^*M}\circ D_{\chi_1}\circ\iota_E\circ\pr_E\circ D_{\chi_2}(e,0)\rangle
\end{split}
\end{equation*}
for $X\in\mx(M)$ and $e\in\Gamma(E)$. Since $
\ldr{X_1}(\pr_{T^*M}D_{\chi_2}(e,0))$ equals\linebreak $\pr_{T^*M}
D_{\chi_1}(0,\pr_{T^*M}D_{\chi_2}(e,0))$ and $\pr_{T^*M}\circ
D_{\chi_1}\circ\iota_E\circ\pr_E\circ D_{\chi_2}(e,0)$ equals\linebreak
$\pr_{T^*M}D_{\chi_1}(\pr_E D_{\chi_2}(e,0),0)$, we find that the
first and fourth term add up to $\langle X,
\pr_{T^*M}D_{\chi_1}D_{\chi_2}(e,0)\rangle$. Similarly the second and third term add up to \linebreak
$-\langle X,
\pr_{T^*M}D_{\chi_2}D_{\chi_1}(e,0)\rangle$ and we get
\[\langle (d_{\chi_1}\phi_{\chi_2}^*-d_{\chi_2}\phi_{\chi_1}^*)^*(e),X\rangle
=\langle \pr_{T^*M}[D_{\chi_1},D_{\chi_2}](e,0), X\rangle.
\]
The proof of the second identity is left to the reader.
\end{proof}

\section{On the proof of Theorem \ref{main2}}\label{proofs2}

Recall that $\mathcal D$ has the following property:
\begin{equation}\label{D_boring}
\mathcal D_{(X,\varepsilon)}(e,\theta)=\mathcal D_{(X,\varepsilon)}(e,0)+(0,\ldr{X}\theta)
\end{equation}
for all $(X,\varepsilon)\in\Gamma(TM\oplus E^*)$ and
$(e,\theta)\in\Gamma(E\oplus T^*M)$.
\eqref{D_boring} and the definition of $\delta$ yield together
\begin{equation}\label{d_D_intertwined}
\delta\circ \pr_E=\pr_E\circ\mathcal D.
\end{equation}
We will use the following lemma.
\begin{lemma}\label{jacobi_eq}
  $\lb\cdot\,,\cdot\rb$ satisfies the Jacobi identity in Leibniz form
  if and only if
\begin{enumerate}
\item $\left[\delta_{\nu_1}, \delta_{\nu_2}\right]=\delta_{\lb\nu_1,\nu_2\rb}$ and 
\item $\pr_{T^*M}[\mathcal D_{\nu_1}, \mathcal D_{\nu_2}]\circ\iota_E=\pr_{T^*M}\circ\mathcal D_{\lb\nu_1,\nu_2\rb}\circ\iota_E$
\end{enumerate}
for all $\nu_1,\nu_2\in\Gamma(TM\oplus E^*)$.
\end{lemma}

\begin{proof}
First note that by \eqref{d_D_intertwined}, we have
\begin{equation}\label{delta_der}
[\delta_{\nu_1},\delta_{\nu_2}]=\pr_E\circ [\mathcal D_{\nu_1},\mathcal D_{\nu_2}]\circ\iota_E.
\end{equation}
If $\lb\cdot\,,\cdot\rb$ satisfies the Jacobi identity in Leibniz form, then (1) and (2) are immediate by
\eqref{eq_jac_dual}.

Conversely, (1) and (2) give using \eqref{delta_der}: $ [\mathcal
D_{\nu_1},\mathcal D_{\nu_2}]\circ\iota_E=\mathcal
D_{\lb\nu_1,\nu_2\rb}\circ\iota_E$.  We have always $[\mathcal
D_{\nu_1},\mathcal
D_{\nu_2}](0,\theta)=(0,\ldr{X_1}\ldr{X_2}\theta-\ldr{X_2}\ldr{X_1}\theta)
=(0,\ldr{[X_1,X_2]}\theta)=\mathcal D_{\lb\nu_1,\nu_2\rb}(0,\theta)$
for all $\theta\in\Omega^1(M)$. This shows that (1), (2) are
equivalent to $[\mathcal D_{\nu_1},\mathcal D_{\nu_2}]=\mathcal
D_{\lb\nu_1,\nu_2\rb}$, which dualises to the Jacobi identity in
Leibniz form for $\lb\cdot\,,\cdot\rb$.
\end{proof}

Now we can prove Theorem  \ref{super}.

\begin{proof}[Proof of Theorem~\ref{super}]We write
$\tau=(e,\theta)$, $\tau_i=(e_i,\theta_i)$ and $\nu=(X,\varepsilon)$,
$\nu_i=(X_i,\varepsilon_i)$ for $i=1,2$.
  By \eqref{prop_2_chi_eq}, we have 
\begin{equation*}
\begin{split}
  & \lb\Xi(\nu_1),\Xi(\nu_2)\rb = \left(\reallywidehat{[\delta_{\nu_1},
      \delta_{\nu_2}]},\dr\ell_{\pr_{E^*}\mathcal
      D_{\nu_1}^*\nu_2}-\reallywidetilde{\pr_{T^*M}\circ[\mathcal
      D_{\nu_1},\mathcal D_{\nu_2}]\circ\iota_E} \right).
\end{split}
\end{equation*}
By Lemma \ref{jacobi_eq}, this is 
\begin{equation*}
\begin{split}
  & \lb\Xi(\nu_1),\Xi(\nu_2)\rb = \left(\reallywidehat{\delta_{\lb\nu_1,\nu_2\rb}},\dr\ell_{\pr_{E^*}\mathcal
      D_{\nu_1}^*\nu_2}-\reallywidetilde{\pr_{T^*M}\circ\mathcal D_{\lb\nu_1,\nu_2\rb}\circ\iota_E} \right)
\end{split}
\end{equation*}
if and only if $\lb\cdot\,,\cdot\rb$ satisfies the Jacobi identity in
Leibniz form. Since $\mathcal
      D_{\nu_1}^*\nu_2=\lb\nu_1,\nu_2\rb$, we are done.
\end{proof}

\section{A non-local Leibniz algebroid}\label{non-local} 
  Let $M=\mathbb S^1 \times \mathbb S^1 \simeq\mathbb T^2$ and consider the vector bundle $\bar{E}= T^*M \oplus
  \wedge^2 T^*M$ over $M$.  Let $\eta\in\Omega^1(\mathbb S^1)$ be the standard
  volume form on the circle and set $\eta_x=\pr_1^*\eta$ and
  $\eta_y=\pr_2^*\eta$, where $\pr_i\colon\mathbb S^1 \times \mathbb
  S^1\to\mathbb S^1$ are the projections, $i=1,2$. Then $\eta_x \wedge
  \eta_y$ is a volume form on $M$ and $\eta_x, \eta_y \in \Omega^1(M)$
  form a basis of one-forms such that the pullback of $\eta_x$ along
  any $\iota_q\colon \mathbb S^1\hookrightarrow\mathbb S^1 \times \{q\}$
  and the pullback of $\eta_y$ to any $\{p\} \times \mathbb S^1$ are
  the standard volume form on the circle.  Define the following
  operations for integration along the first fibre. For $f,g,h\in
  C^\infty(M)$: $\int_{\mathbb S^1} f \eta_x + g \eta_y \in
  C^\infty(M)$,
\begin{align*} 
  \left(\int_{\mathbb S^1} f \eta_x + g \eta_y\right)(p,q) := \int_{\mathbb S^1} \iota_q^*(f) \eta
\end{align*}
and $\int_{\mathbb S^1} h\, \eta_x \wedge \eta_y\in\Omega^1(M)$, 
\begin{align*}
 \left( \int_{\mathbb S^1} h\, \eta_x \wedge \eta_y\right)(p,q):= \left(\int_{\mathbb S^1} \iota_q^*(h)\eta\right) \eta_y (p,q). 
\end{align*}
Clearly, the resulting function
$\int_{\mathbb S^1} f \eta_x + g \eta_y \in C^\infty(M)$ is constant
along the first $\mathbb S^1$, i.e.~only a function of $q$ in the
notation above. In the same manner, the one-form
$\int_{\mathbb S^1} h\, \eta_x \wedge \eta_y$ is constant along the
first $\mathbb S^1$ and only has a $\eta_y$ component. That is, the
obtained functions and 1-forms are invariant along the fibers of
$\pr_2$.

Now we define a bracket on $\bar{E}=T^*M \oplus \wedge^2 T^*M$ as follows: 
\begin{equation*} 
  \lb (\alpha_1,\alpha_2),(\beta_1,\beta_2)\rb = \left(0,\left(\int_{\mathbb S^1} \alpha_1\right) \beta_2 + \left(\int_{\mathbb S^1} \alpha_2\right) \wedge \beta_1\right) 
\end{equation*} 
and we prove that $(\bar{E}=T^*M \oplus \wedge^2 T^*M, \lb \cdot,\cdot\rb,0\colon\bar{E}\rightarrow TM$ is a Leibniz algebroid. Since the bracket is clearly $C^{\infty}$-linear in the second component and thus satisfies the Leibniz rule for functions with the zero-anchor, it suffices to check the Jacobi identity in Leibniz form. For simplicity, we just write $\int$ for
$\int_{\mathbb S^1}$, and this is always the integration along the first $\mathbb S^1$. We have 
\begin{align*} 
  &\lb (\alpha_1,\alpha_2), \lb (\beta_1, \beta_2), (\gamma_1, \gamma_2) \rb \rb = \left\lb (\alpha_1,\alpha_2), \left(0,\int \beta_1 \,\gamma_2 + \int \beta_2 \wedge \gamma_1\right) \right\rb \\
  &= \left(0, \int \alpha_1 \int \beta_1 \,\gamma_2 + \int \alpha_1 \int \beta_2 \wedge \gamma_1 \right)
\end{align*}
and in a similar manner
\begin{align*}
  \lb \lb (\alpha_1,\alpha_2), ( \beta_1, \beta_2) \rb, (\gamma_1, \gamma_2) \rb = \left\lb \left(0,\int \alpha_1\, \beta_2 + \int \alpha_2 \wedge \beta_1 \right), (\gamma_1,\gamma_2) \right\rb \\
  = \left(0,\int\left(\int \alpha_1 \, \beta_2 + \int \alpha_2 \wedge \beta_1\right) \wedge \gamma_1 \right) \\
  = \left(0,\int \alpha_1 \int\beta_2 \wedge \gamma_1 - \int \beta_1 \int \alpha_2 \wedge \gamma_1\right).
\end{align*}
Therefore we get
\begin{align*}
\lb (\alpha_1,\alpha_2), \lb ( \beta_1, \beta_2), (\gamma_1, \gamma_2) \rb \rb - \lb (\beta_1,\beta_2), \lb ( \alpha_1, \alpha_2), (\gamma_1, \gamma_2) \rb \rb \\
  - \lb \lb (\alpha_1,\alpha_2), (\beta_1, \beta_2) \rb, (\gamma_1, \gamma_2) \rb \\
  = \left(0, \int \alpha_1 \int \beta_1 \gamma_2 + \int \alpha_1\int \beta_2 \wedge \gamma_1 -  \int \beta_1 \int \alpha_1 \gamma_2 - \int \beta_1 \int \alpha_2 \wedge \gamma_1 \right. \\
  - \left. \int\alpha_1 \int \beta_2 \wedge \gamma_1 + \int
    \beta_1 \int \alpha_2 \wedge \gamma_1 \right) = 0
\end{align*}
This Leibniz algebroid is \emph{non-local}, i.e.~its bracket \emph{not} given by a bilinear differential operator of any order.

\bigskip

\noindent\textbf{Acknowledgement:} The authors wish to thank an
anonymous referee for useful comments on an earlier version of this work.

\def\cprime{$'$} \def\polhk#1{\setbox0=\hbox{#1}{\ooalign{\hidewidth
  \lower1.5ex\hbox{`}\hidewidth\crcr\unhbox0}}} \def\cprime{$'$}
  \def\cprime{$'$}

\end{document}